\documentclass[a4paper,11pt]{article}
\usepackage[utf8]{inputenc}
\usepackage{amsmath,amsthm,mathrsfs}
\usepackage{latexsym}
\usepackage{amssymb}
\usepackage[all]{xy}
\usepackage{enumerate}
\usepackage{pst-all}
\usepackage{pstricks}
\usepackage{enumerate}
\usepackage{lipsum}
\usepackage{float}
\input{xy}
\theoremstyle{definition}
 \newtheorem{theo}{Theorem}[section]
 \newtheorem{propo}[theo]{Proposition}
 \newtheorem{lem}[theo]{Lemma}
 \newtheorem{cor}[theo]{Corollary}
 \newtheorem{obs}[theo]{Observation}
 \newtheorem{ejem}[theo]{Example}
 \newtheorem{defin}[theo]{Definition}
\def\L{\Lambda}
\def\G{\Gamma}
\def\o{\mathfrak{o}}

\def\A{\mathscr{A}}
\def\B{\mathscr{B}}
\def\C{\mathscr{C}}
\def\D{\mathscr{D}}
\def\T{\mathscr{T}}
\def\V{\mathscr{V}}
\def\Q{\mathcal{Q}}

\def\occ{\textrm{occ}}
\def\val{\textrm{val}}
\def\CC{\mathscr{CC}}
\def\hom{\textrm{Hom}}
\def\olCC{\mathscr{AA}}
\def\M{\mathscr{M}}

\newcommand\blfootnote[1]{
  \begingroup
  \renewcommand\thefootnote{}\footnote{#1}
  \addtocounter{footnote}{-1}
  \endgroup
}

\newpsobject{grilla}{psgrid}{subgriddiv=1,griddots=10,gridlabels=6pt}
\title{The Dimension of the Center of a Brauer Configuration Algebra}
\author{Alex Sierra C.}
\date{}
\begin{document}

\maketitle

\begin{abstract}
We determine the dimension of the center of a Brauer configuration algebra in terms of combinatorial data from a Brauer configuration.
\end{abstract}
\blfootnote{
\indent This paper is part of my PhD thesis which was defended in IME-USP with support of a schollarship of PROEX-CAPES. This was done with the supervision of my adviser E. N. Marcos and my co-adviser S. Schroll.}
\section{Introduction}
Brauer configuration algebras were recently added to the mathematical literature by E. Green and S. Schroll in \cite{brau} as a generalization of Brauer graph algebras. As the authors mention in \cite[Introduction]{brau} this class of algebras is a new class of mostly wild algebras whose additional structure arises from a combinatorial data, called Brauer configuration. Brauer configuration algebras are a generalization of Brauer graph algebras, in the sense that every Brauer graph is a Brauer configuration and every Brauer graph algebra is a Brauer configuration algebra. 

Brauer graph algebras are so well-studied and understood because the combinatorial data of the underlying Brauer graph encodes many aspects of the representation theory of a Brauer graph algebra. For example, the Yoneda algebra \cite{green2}, or group actions and coverings \cite{green1}, just to mention a few. So, the expectation is that the Brauer configuration will encode the representation theory of Brauer configuration algebras.%%--- Esta parte necesita ampliarse un poco más---%%

%Brauer configuration $\G=(\G_0,\G_1,\mu,\o)$ and its induced Brauer configuration algebra $\L$ we are going to consider the space $v\L v$, where $V\in\G_1$, and calculate the whole basis of this space.

\section{Preliminaries and Notation}\label{secc00b}
We will give a quick review of the definition of a Brauer configuration and its associated Brauer configuration algebra. We refer the reader to \cite[Sections 1 and 2]{brau} for complete details and many examples.

We begin with a tuple $\G=(\G_0,\G_1,\mu)$ where $\G_0$ is a finite set whose elements we call \textit{vertices}, $\mu:\G_0\to\mathbb{Z}_{>0}$ a function, called a \textit{multiplicity function} of $\G$, and $\G_1$ is a finite collection of labeled finite sets of vertices where repetitions are allowed. We can also say that $\G_1$ is a finite collection of finite labeled multisets whose elements are in $\G_0$. We call the elements of $\G_1$ \textit{polygons}. Given a polygon $V$, we call its elements the \textit{vertices of} $V$. If $V\in\G_1$ and $\alpha\in\G_0$, define $\occ(\alpha,V)$ to be the number of times $\alpha$ occurs as a vertex in $V$ and define the \textit{valency} of $\alpha$, denoted by $\val(\alpha)$, to be the value $\sum\limits_{V\in\G_1}\occ(\alpha,V)$.

An \textit{orientation for} $\G$ is a choice, for each vertex $\alpha\in\G_0$, of a cyclic ordering of the polygons in which $\alpha$ occurs as a vertex, including repetitions. To be more precise, for each $\alpha\in\G_0$ such that $\val(\alpha)=t>1$ or $\mu(\alpha)>1$, let $V_1,\ldots, V_t$ be the list of polygons in which $\alpha$ occurs as a vertex, with a polygon $V$ occurring $\occ(\alpha,V)$ times in the list, that is $V$ occurs the number of times $\alpha$ occurs as a vertex in $V$. The cyclic order at vertex $\alpha$ is obtained by linearly ordering the list, say $V_{i_1}<\cdots<V_{i_t}$ and by adding $V_{i_t}<V_{i_1}$. We observe that any cyclic permutation of a chosen cyclic ordering at vertex $\alpha$ can represent the same ordering. That is, if $V_1<\cdots<V_t$ is the chosen cyclic ordering at vertex $\alpha$, so is a cyclic permutation such as $V_2<V_3<\cdots<V_t<V_1$ or $V_3<V_4<\cdots<V_t<V_1<V_2$.

\begin{defin}\label{defbra}
 A \textit{Brauer configuration} is a tuple $\G=(\G_0,\G_1,\mu,\o)$, where $\G_0$ is a set of vertices, $\G_1$ a set of polygons, $\mu$ is a multiplicity function, and $\o$ is an orientation for $\G$, such that the following conditions hold
 \begin{enumerate}
  \item[C1.] Every vertex in $\G_0$ is a vertex in at least one polygon in $\G_1$.
  \item[C2.] Every polygon in $\G_1$ has at least two vertices (which can be the same).
  \item[C3.] Every polygon in $\G_1$ has at least one vertex $\alpha$ such that $\val(\alpha)\mu(\alpha)>1$.
 \end{enumerate}
\end{defin}

%Let $\G=(\G_0,\G_1,\mu,\o)$ be a Brauer Configuration, $\Q_{\G}$ its induced quiver and $\L_{\G}=K\Q_{\G}/I_{\G}$ the Brauer configuration algebra associated to $\G$. When the configuration is clear from the context we write $\L$ instead of $\L_{\G}$.

We now define some special sets formed by vertices of a Brauer configuration. They will be used to simplify the notation in the process of the computations. For the Brauer configuration $\G$ we define the following special subsets of $\G_0$.

\begin{eqnarray*}
 \T_{\G} & = & \{\,\alpha\in\G_0\,|\,\mu(\alpha)\val(\alpha)=1\,\},\\
 \C_{\G} & = & \{\,\alpha\in\G_0\,|\,\val(\alpha)=1\textrm{ and }\mu(\alpha)>1\,\},\\
 \D_{\G} & = & \{\,\alpha\in\G_0\,|\,\val(\alpha)>1\,\}.
\end{eqnarray*}

Any element of $\T_{\G}$ is called a \textit{truncated vertex} of $\G$ (see \cite[Definition 1.3]{brau}). Thus, we call the set  $\G_0\setminus\T_{\G}$ the collection of the \textit{non-truncated vertices}. We can decompose the set $\D_{\G}$ into the following sets.
\begin{eqnarray*}
 \A_{\G} & = & \{\,\alpha\in\D_{\G}\,|\,\mu(\alpha)>1\,\},\\
 \B_{\G} & = & \{\,\alpha\in\D_{\G}\,|\,\mu(\alpha)=1\,\}.
\end{eqnarray*}

It is clear that $\G_0\setminus\T_{\G}=\C_{\G}\dot{\cup}\D_{\G}=\A_{\G}\dot{\cup}\B_{\G}\dot{\cup}\C_{\G}$.\\

Let $\G=(\G_0,\G_1,\mu,\o)$ be a Brauer configuration. For $\alpha\in\G_0\setminus\T_{\G}$ consider the list of polygons $V$ containing $\alpha$ such that $V$ occurs in this list $\occ(\alpha,V)$ times. We know that the orientation $\o$ provides a cyclic ordering of this list. We call such a cyclically ordered list the \textit{successor sequence at} $\alpha$. Suppose that $\alpha\in\D_{\G}$ and that $V_1<\cdots<V_t$ is the successor sequence at $\alpha$ (here $\val(\alpha)=t$). Then we say that $V_{i+1}$ is the \textit{successor of} $V_i$ \textit{at} $\alpha$, for $1\le i\le t$, where $V_{t+1}=V_1$. Note that for the case $\alpha\in\C_{\G}$ there exists a unique polygon $V$ such that $\alpha\in V$, then the successor sequence at $\alpha$ is just $V$.\\

Given $\G=(\G_0,\G_1,\mu,\o)$ we define the quiver $\Q_{\G}$ as follows. The vertex set $\{\,v_1,\ldots,v_m\,\}$ of $\Q_{\G}$ is in a bijection with the set of polygons $\{\,V_1,\ldots,V_m\,\}$ in $\G_1$, noting that there is exactly one vertex in $\Q_{\G}$ for every polygon in $\G_1$. We call $v_i$ (resp. $V_i$) the vertex (resp. polygon) \textit{associated} to $V_i$ (resp. $v_i$). In order to define the arrows in $\Q_{\G}$, we use the successor sequences. For each $\alpha\in\G_0\setminus\T_{\G}$, and each successor $V'$ of $V$ at $\alpha$, there is an arrow from $v$ to $v'$ in $\Q_{\G}$, where $v$ and $v'$ are the vertices in $\Q_{\G}$ associated to the polygons $V$ and $V'$ in $\G_1$, respectively.\\

We should note that $V'$ can be a successor of $V$ more than once at a given vertex of $\G_0$. We also note that $V'$ can be a  successor of $V$ at more than one vertex of $\G_0$, and for each such occurrence there is an arrow from $v$ to $v'$. In particular, $\Q_{\G}$ may have multiple arrows from $v$ to $v'$. We will see an example of this (Example \ref{ej01}).\\

For each $\alpha\in\D_{\G}$ with successor sequence $V_{i_1}<\cdots<V_{i_{\val(\alpha)}}$ we have a corresponding collection of arrows in $\Q_{\G}$

\begin{equation}\label{001}
v_{i_1}\stackrel{a^{(\alpha)}_{j_1}}{\longrightarrow} v_{i_2} \stackrel{a^{(\alpha)}_{j_2}}{\longrightarrow} \cdots \stackrel{a^{(\alpha)}_{j_{\val(\alpha)-1}}}{\longrightarrow}v_{i_{\val(\alpha)}}\stackrel{a^{(\alpha)}_{j_{\val(\alpha)}}}{\longrightarrow} v_{i_1}
\end{equation}
Let $C_l=a^{(\alpha)}_{j_{l}}a^{(\alpha)}_{j_{l+1}}\cdots a^{(\alpha)}_{j_{\val(\alpha)}}a^{(\alpha)}_{j_{1}}\cdots a^{(\alpha)}_{j_{l-1}}$ be the oriented cycle in $\Q_{\G}$, for $1\le l\le\val(\alpha)$. We call any of these cycles a \textit{special} $\alpha$-\textit{cycle}. We observe that when $\alpha\in\C_{\G}$, we have just one special $\alpha$-cycle, which is a loop at the vertex in $\Q_{\G}$ associated to the unique polygon containing $\alpha$. Now fix a polygon $V$ in $\G_1$ and suppose that $\occ(\alpha,V)=t\ge1$. Then there are $t$ indices $l_1,\ldots,l_t$ such that $V=V_{i_{l_r}}$ for every $1\le r\le t$. We call any of the cycles $C_{l_1},\ldots,C_{l_t}$ a \textit{special} $\alpha$-\textit{cycle at} $v$, and we denote the collection of these cycles in $\Q_{\G}$ by $\C_{(\alpha)}^{\,v}$. Note that $|\C_{(\alpha)}^{\,v}|=\occ(\alpha,V)$. And if we define $\V_{(\alpha)}=\{\,V\in\G_1\,|\,\alpha\in V\,\}$, the set of polygons containing $\alpha$, we define \begin{equation}\label{038}\CC_
{(\alpha)}:=\bigcup_{V\in\V_{(\alpha)}}\C_{(\alpha)}^{\,v}.\end{equation} This is the collection of all  special $\alpha$-cycles. For the particular case $\alpha\in\C_{\G}$ we define $\CC_{(\alpha)}$ as the set 
consisting of the unique loop associated to the vertex $\alpha$, and by $\V_{(\alpha)}$ the set consisting of the only polygon containing $\alpha$.\\
 From these observations and definitions we have that for any $\alpha\in\G_0\setminus\T_{\G}$ \begin{eqnarray}|\CC_{(\alpha)}| & = & \sum_{V\in\V_{(\alpha)}}\occ(\alpha,V)\nonumber\\ & = & \sum_{V\in\G_1}\occ(\alpha,V)\nonumber\\ & = & \val(\alpha)\label{039}.\end{eqnarray} Before giving the definition of a Brauer configuration algebra we need one last definition. Let $\CC$ be the set of special cycles in $\Q_{\G}$, that is \begin{equation}\label{040}\CC:=\bigcup_{\alpha\in\G_0\setminus\T_{\G}}\CC_{(\alpha)}\end{equation} and let $f:\CC\to(\Q_{\G})_1$ be the map which sends a special cycle to its first arrow.
\begin{propo}\label{pr005}
 Let $\G=(\G_0,\G_1,\mu,\o)$ be a Brauer configuration with associated quiver $\Q$. Then \[|\Q_1|=\sum_{\alpha\in\G_0\setminus\T_{\G}}\val(\alpha),\] and for each $a\in\Q_1$ there exists a unique $\alpha\in\G_0\setminus\T_{\G}$ such that $a$ occurs once in any of the special cycles in $\CC_{(\alpha)}$.
\end{propo}

\begin{proof}
 By the properties stated for special cycles in \cite[Section 2.3]{brau} we have
 \begin{itemize}
  \item $\left\{\CC_{(\alpha)}\,|\,\alpha\in\G_0\setminus\T_{\G}\right\}$ is a partition of $\CC$;
  \item the map $f$ defined above is a bijective function.
 \end{itemize}
Then by the bijection of $f$ and the equalities in (\ref{039}) and (\ref{040}), we obtain $|\Q_1|=\sum_{\alpha\in\G_0\setminus\T_{\G}}\val(\alpha)$. This value also coincides with the number of special cycles. The second assertion follows from the bijection of $f$ and the fact that any two special cycles in $\CC_{(\alpha)}$ are cyclic permutations of one another (see (F'3) and (F'4) in \cite[Section 2.3]{brau}).
\end{proof}

\subsection{Definition of a Brauer configuration algebra}
We define a set of elements $\rho_{\G}$ in $K\Q_{\G}$ which will generate the ideal of relations $I_{\G}$ of the Brauer configuration algebra associated to the Brauer configuration $\G$. The set of relations $\rho_{\G}$ is defined by the following three types of relations:\\\\
\textit{Relations of type one.} It is the subset of $K\Q_{\G}$ \[\bigcup_{V\in\G_1}\left(\bigcup_{\alpha,\beta\in V\setminus\T_{\G}}\left\{\,C^{\mu(\alpha)}-D^{\mu(\beta)}\,|\,C\in\C_{(\alpha)}^{\,v},D\in\C_{(\beta)}^{\,v}\,\right\}\right).\]
\\
\textit{Relations of type two.} It is the subset of $K\Q_{\G}$ \[\bigcup_{\alpha\in\G_0\setminus\T_{\G}}\left\{\,C^{\mu(\alpha)}f(C)\,|\,C\in\CC_{(\alpha)}\,\right\}.\]
\\
\textit{Relations of type three.} It is the set of all quadratic monomial relations of the form $ab$ in $K\Q_{\G}$ where $ab$ is not a subpath of any special cycle.
\begin{defin}
 Let $K$ be a field and $\G$ a Brauer configuration. The \textit{Brauer configuration algebra} $\L_{\G}$ \textit{associated to} $\G$ is defined to be $K\Q_{\G}/I_{\G}$, where $\Q_{\G}$ is the quiver associated to $\G$ and $I_{\G}$ is the ideal in $K\Q_{\G}$ generated by the set of relations $\rho_{\G}$ of type one, two and three.
\end{defin}

We now refer the reader to \cite[page 545]{brau} to find the definition of a \textit{connected} Brauer configuration. 

\begin{propo}
 Let $\L$ be the Brauer configuration algebra associated to the connected Brauer configuration $\G$. Then rad$^2(\L)\neq0$.
\end{propo}

\begin{proof}
Let $\G=(\G_0,\G_1,\mu,\o)$ be a connected Brauer configuration, and let $\L=\Q/I$ be the induced Brauer configuration algebra. By \cite[Proposition 3.5]{brau} we have that $\L$ is an indecomposable algebra. Now, we assume that $\val(\alpha)<2$, for all $\alpha\in\G_0$. Because of the connectedness of $\G$ we have that necessarily $\G_1=\left\{V\right\}$, then the induced quiver $\Q$ has an only one vertex $v$, which is the vertex associated to $V$. Then $\Q$ has the form 
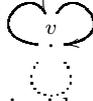
\begin{equation*}\xymatrix{\stackrel{v}{\cdot}\ar@(l,u)\ar@(u,r)\ar@(dr,dl)@{..}}\end{equation*}
where the number of loops at $v$ coincides with the number of nontruncated vertices in $V$. By condition C3 of Definition \ref{defbra} there exists $\alpha\in V$ such that $\mu(\alpha)\ge2$. If $a$ is the loop in $\Q$ associated to $\alpha$, then the class of $a^2$ in $\L$ is a non zero element of rad$^2(\L)$ by \cite[Proposition 3.3]{brau}.

Now, we assume that there exists a vertex $\alpha\in\G_0$ such that $\val(\alpha)\ge2$. Then we have that any special $\alpha$-cycle in $\Q$ contains at least two consecutive arrows. Hence by \cite[Proposition 3.3]{brau} we have again that  rad$^2(\L)\neq0$.
\end{proof}

\section{Basis in $v\L v$.}
%It is not difficult to check that
%\begin{eqnarray*}
 %|\CC_{(\alpha)}| & = & \sum_{V\in\V_{(\alpha)}}|\C_{(\alpha)}^{\,v}|\\ & = & \sum_{V\in\V_{(\alpha)}}\occ(\alpha,V)\\ & = & %\sum_{V\in\G_1}\occ(\alpha,V)\\ & = & \val(\alpha)
%\end{eqnarray*}

%it is important to keep the notation in the arrows associated to one vertex

%As we will see further will be important the set of vertices in the configuration that appear more than once in any of the polygons where they belong. So for $V\in\G_1$ let $\A_{\G}^{V>1}$ and $\B_{\G}^{V>1}$ be the subsets of $\G_0$  defined by
%\begin{eqnarray*}
 %\A_{\G}^{V>1} & = & \{\,\alpha\in\A_{\G}\cap\overline{V}\,|\,\occ(\alpha,V)>1\,\},\\
 %\B_{\G}^{V>1} & = & \{\,\beta\in\B_{\G}\cap\overline{V}\,|\,\occ(\beta,V)>1\,\}.
%\end{eqnarray*}

%We can say about the previous sets that the polygon $V\in\G_1$ is \textit{self-folded} if either $\A_{\G}^{V>1}\neq\emptyset$ or $\B_{\G}^{V>1}\neq\emptyset$ (see \cite[Section 1]{brau} to see definition of self-folded polygon).\\
Let $\G$ be a Brauer configuration and $\L$ the induced Brauer configuration algebra. Before giving the set of elements of a $K$-basis of $v\L v$ for the vertex $v$ in $\Q$ associated to the polygon $V$, we introduce a simple graphical method which allows us to recognize all the possible elements, in terms of a basis, that live in $v\L v$. We start with an example

\begin{ejem}\label{ej01}
 Let $\G=(\G_0,\G_1,\mu,\o)$ be a Brauer configuration and let $\Q$ be its induced quiver. We suppose that in the configuration $\G$ we have a non-trucated vertex $\alpha$ which has a successor sequence given by \[\alpha: V<V<V<V<V<W<W<W<W<W.\] We see at this case that $\occ(\alpha,V)=\occ(\alpha,W)=5$. In particular, we also see that, for example, $V$ is the successor of $V$ at $\alpha$ four times. So having in mind this successor sequence we can say that attached to the quiver $\Q$ we will have the subquiver.
 
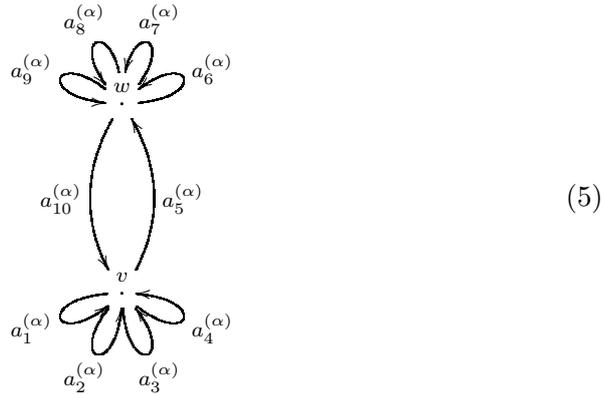
\begin{equation}\label{037}\begin{split}\xymatrix{\stackrel{w}{\cdot}\ar@(r,ur)_{a^{(\alpha)}_6}\ar@(ur,u)_{a^{(\alpha)}_7}\ar@(u,ul)_{a^{(\alpha)}_8}\ar@(ul,l)_{a^{(\alpha)}_9}\ar@/_1pc/[dd]_{a^{(\alpha)}_{10}} \\ \\ \stackrel{v}{\cdot}\ar@(l,dl)_{a^{(\alpha)}_1}\ar@(dl,d)_{a^{(\alpha)}_2}\ar@(d,dr)_{a^{(\alpha)}_3}\ar@(dr,r)_{a^{(\alpha)}_4}\ar@/_1pc/[uu]_{a^{(\alpha)}_5} }\end{split}\end{equation}
 It is easy to see the special $\alpha$-cycles at $v$ and $w$. But there are many more elements that are in the space $v\L v$ and they are not special cycles. This is not easily seen from the previous graphical expression. So the idea is to look at this quiver in a different way. We open up the entire cycle like this.
 
 \begin{figure}[H]
 \centering
 \begin{pspicture}(-2.8,-2.8)(2.8,2.8)%\grilla
  \psarc[linewidth=1.1pt]{->}(0,0){2}{0}{36}
  \psarc[linewidth=1.1pt]{->}(0,0){2}{36}{72}
  \psarc[linewidth=1.1pt]{->}(0,0){2}{72}{108}
  \psarc[linewidth=1.1pt]{->}(0,0){2}{108}{144}
  \psarc[linewidth=1.1pt]{->}(0,0){2}{144}{180}
  \psarc[linewidth=1.1pt]{->}(0,0){2}{180}{216}
  \psarc[linewidth=1.1pt]{->}(0,0){2}{216}{252}
  \psarc[linewidth=1.1pt]{->}(0,0){2}{252}{288}
  \psarc[linewidth=1.1pt]{->}(0,0){2}{288}{324}
  \psarc[linewidth=1.1pt]{->}(0,0){2}{324}{360}
  \rput[bl](1.902,0.618){{\footnotesize $a^{(\alpha)}_3$}}
  \rput[bl](1.176,1.618){\footnotesize $a^{(\alpha)}_4$}
  \rput[b](0,2){\footnotesize$a^{(\alpha)}_5$}
  \rput[br](-1.176,1.618){\footnotesize$a^{(\alpha)}_6$}
  \rput[br](-1.902,0.618){\footnotesize$a^{(\alpha)}_7$}
  \rput[tr](-1.902,-0.618){\footnotesize$a^{(\alpha)}_8$}
  \rput[tr](-1.176,-1.618){\footnotesize$a^{(\alpha)}_{9}$}
  \rput[t](0,-2){\footnotesize$a^{(\alpha)}_{10}$}
  \rput[tl](1.176,-1.618){\footnotesize$a^{(\alpha)}_1$}
  \rput[tl](1.902,-0.618){\footnotesize$a^{(\alpha)}_2$}
  \rput[tl](0.618,-1.902){\large$v$}
  \rput[tl](1.618,-1.176){\large$v$}
  \rput[l](2,0){\large$v$}
  \rput[bl](1.696,1.040	){\large$v$}
  \rput[bl](0.618,1.902){\large$v$}
  \rput[br](-0.618,1.902){\large$w$}
  \rput[br](-1.618,1.176){\large$w$}
  \rput[r](-2,0){\large$w$}
  \rput[tr](-1.618,-1.176){\large$w$}
  \rput[tr](-0.618,-1.902){\large$w$}
 \end{pspicture}
 \caption{}\label{fig1}
\end{figure}

If we are only interested to look for oriented cycles at the vertex $v$ of the quiver, we simply erase the $w$'s and consider the $v$'s. 

\begin{figure}[H]
\centering
 \begin{pspicture}(-2.8,-2.8)(2.8,2.8)%\grilla
  \psarc[linewidth=1.1pt]{->}(0,0){2}{0}{36}
  \psarc[linewidth=1.1pt]{->}(0,0){2}{36}{72}
  \psarc[linewidth=1.1pt]{->}(0,0){2}{72}{108}
  \psarc[linewidth=1.1pt]{->}(0,0){2}{108}{144}
  \psarc[linewidth=1.1pt]{->}(0,0){2}{144}{180}
  \psarc[linewidth=1.1pt]{->}(0,0){2}{180}{216}
  \psarc[linewidth=1.1pt]{->}(0,0){2}{216}{252}
  \psarc[linewidth=1.1pt]{->}(0,0){2}{252}{288}
  \psarc[linewidth=1.1pt]{->}(0,0){2}{288}{324}
  \psarc[linewidth=1.1pt]{->}(0,0){2}{324}{360}
  \rput[bl](1.902,0.618){{\footnotesize $a^{(\alpha)}_3$}}
  \rput[bl](1.176,1.618){\footnotesize $a^{(\alpha)}_4$}
  \rput[b](0,2){\footnotesize$a^{(\alpha)}_5$}
  \rput[br](-1.176,1.618){\footnotesize$a^{(\alpha)}_6$}
  \rput[br](-1.902,0.618){\footnotesize$a^{(\alpha)}_7$}
  \rput[tr](-1.902,-0.618){\footnotesize$a^{(\alpha)}_8$}
  \rput[tr](-1.176,-1.618){\footnotesize$a^{(\alpha)}_{9}$}
  \rput[t](0,-2){\footnotesize$a^{(\alpha)}_{10}$}
  \rput[tl](1.176,-1.618){\footnotesize$a^{(\alpha)}_1$}
  \rput[tl](1.902,-0.618){\footnotesize$a^{(\alpha)}_2$}
  \rput[tl](0.618,-1.902){\large$v$}
  \rput[tl](1.618,-1.176){\large$v$}
  \rput[l](2,0){\large$v$}
  \rput[bl](1.696,1.040	){\large$v$}
  \rput[bl](0.618,1.902){\large$v$}
 \end{pspicture}
 \caption{}\label{fig2}
\end{figure}
Now consider the following elements \[q^{(\alpha,v)}_1=a^{(\alpha)}_1,\,\,q^{(\alpha,v)}_2=a^{(\alpha)}_2,\,\,q^{(\alpha,v)}_3=a^{(\alpha)}_3,\]\[q^{(\alpha,v)}_4=a^{(\alpha)}_4,\,\,q^{(\alpha,v)}_5=a^{(\alpha)}_5a^{(\alpha)}_6\cdots a^{(\alpha)}_{10}\] Each of this elements is an oriented cycle at $v$. By defining $q^{(\alpha,v)}_6=q^{(\alpha,v)}_1$ and reducing subindices modulo $\occ(\alpha,V)=5$, we also have that  all elements in the set \[\bigcup_{r=1}^4\left\{\,q^{(\alpha,v)}_l\cdots q^{(\alpha,v)}_{l+r-1}\,|\,1\le l\le5\,\right\}\] are oriented cycles at $v$. As we can see there are many oriented cycles at the vertex $v$ which are not special cycles. We can do the same to find oriented cycles for the vertex $w$. These ideas can be generalized.
\end{ejem}

\subsection{Non-special cycles}\label{subsecc00c}

We see in \cite[Proposition 3.3]{brau} that for a Brauer configuration algebra $\L$ induced by the configuration $\G$, the elements of a basis of $\L$ are all prefixes of the elements $C^{\mu(\alpha)}$, where $C\in\CC_{(\alpha)}$ and $\alpha\in\G_0\setminus\T_{\G}$. We will now compute a vector space basis for $v\L v$, where $v$ is the vertex in the induced quiver associated to the polygon $V$.\\

For $\alpha\in\D_{\G}$ fixed let $V_{i_1}<\cdots<V_{i_{\val(\alpha)}}$ be its successor sequence, then in the quiver $\Q$ we have a sequence of arrows

\begin{equation}\label{002}
v_{i_1}\stackrel{a^{(\alpha)}_{j_1}}{\longrightarrow} v_{i_2} \stackrel{a^{(\alpha)}_{j_2}}{\longrightarrow} \cdots \stackrel{a^{(\alpha)}_{j_{\val(\alpha)-1}}}{\longrightarrow}v_{i_{\val(\alpha)}}\stackrel{a^{(\alpha)}_{j_{\val(\alpha)}}}{\longrightarrow} v_{i_1}.
\end{equation}

This sequence can be represented in a cyclic way. The sequentially composition of these arrows is a oriented cycle in $\Q$. Let's put the arrows and vertices in a cyclic drawing, like this.
\begin{figure}[H]
\centering
 \begin{pspicture}(-2.8,-2.8)(2.8,2.8)%\grilla
  \psarc[linewidth=1.1pt]{->}(0,0){2}{-110}{-90}
  \psarc[linewidth=1.1pt]{->}(0,0){2}{-90}{-70}
  \psarc[linewidth=1.1pt]{->}(0,0){2}{-70}{-50}
  \psarc[linewidth=0.7pt,linestyle=dashed]{->}(0,0){2}{-50}{40}
  \psarc[linewidth=1.1pt]{->}(0,0){2}{40}{60}
  \psarc[linewidth=1.1pt]{->}(0,0){2}{60}{80}
  \psarc[linewidth=0.7pt,linestyle=dashed]{->}(0,0){2}{80}{250}
  \rput[t](0.347,-1.970){\footnotesize$a^{(\alpha)}_{j_1}$}
  \rput[tl](1,-1.732){\footnotesize$a^{(\alpha)}_{j_2}$}
  \rput[t](-0.347,-1.970){\footnotesize$a^{(\alpha)}_{j_{\val(\alpha)}}$}
  \rput[bl](1.285,1.532){\footnotesize$a^{(\alpha)}_{j_{l}}$}
  \rput[bl](0.684,1.879){\footnotesize$a^{(\alpha)}_{j_{l+1}}$}
  \rput[b](0,-2){\footnotesize$v_{i_1}$}
  \rput[br](0.684,-1.879){\footnotesize$v_{i_2}$}
  \rput[br](1.285,-1.532){\footnotesize$v_{i_3}$}
  \rput[tr](1.532,1.285){\footnotesize$v_{i_l}$}
  \rput[tr](1,1.732){\footnotesize$v_{i_{l+1}}$}
 \end{pspicture}
 \caption{}\label{fig3}
\end{figure}
Now for a polygon $V\in\G_1$ we suppose that $\occ(\alpha,V)>1$ then there are indices $l_1,\ldots,l_{\occ(\alpha,V)}$ such that $v=v_{i_{l_t}}$ for each $1\le t\le\occ(\alpha,V)$, and where $v$ is the vertex in $\Q$ associated to $V$. From Figure \ref{fig3} we are going to derive another one. So, in Figure \ref{fig3} we delete the vertices that are not equal to $v$, as also the arrows. We obtain the following figure.
\begin{figure}[H]
\centering
 \begin{pspicture}(-2.8,-2.8)(2.8,2.8)%\grilla
  \psarc[linewidth=1.1pt]{->}(0,0){2}{-110}{-75}
  \psarc[linewidth=1.1pt]{->}(0,0){2}{-75}{-45}
  \psarc[linewidth=1.1pt]{->}(0,0){2}{-45}{5}
  \psarc[linewidth=0.7pt,linestyle=dashed]{->}(0,0){2}{5}{110}
  \psarc[linewidth=1.1pt]{->}(0,0){2}{110}{140}
  \psarc[linewidth=0.7pt,linestyle=dashed]{->}(0,0){2}{140}{250}
  \rput[br](0.518,-1.932){\small$v$}
  \rput[br](1.414,-1.414){\small$v$}
  \rput[r](1.992,0.174){\small$v$}
  \rput[tl](-0.684,1.879){\small$v$}
  \rput[tl](-1.532,1.285){\small$v$}
 \end{pspicture}
 \caption{}\label{fig5}
 \end{figure}

Now fix one of the vertices $v$, any vertex you want, and label it as \textit{1st} $v$; in counter clockwise consider the next one $v$ and label it as \textit{2nd} $v$, and then the next one and label it as \textit{3rd} $v$, and so on. At each occurrence of these labeled $v$'s we define two types of oriented cycles.

\begin{itemize}
 \item At the \textit{1st} $v$ we have $C^{(\alpha,v)}_1$ as the special $\alpha$-cycle at $v$ corresponding to the  \textit{1st} $v$, and $q^{(\alpha,v)}_1$, the oriented cycle obtained by composing all the arrows between the \textit{1st} $v$ and the \textit{2nd} $v$.
 \item At the \textit{2nd} $v$ we have $C^{(\alpha,v)}_2$ as the special $\alpha$-cycle at $v$ corresponding to the \textit{2nd} $v$, and $q^{(\alpha,v)}_2$, the oriented cycle obtained by composing all the arrows between the \textit{2nd} $v$ and the \textit{3rd} $v$.
 \item And so on $,\ldots$
\end{itemize}

What we obtain after this finite procedure is the following figure.
\begin{figure}[H]
\centering
 \begin{pspicture}(-2.8,-2.8)(2.8,2.8)%\grilla
  \psarc[linewidth=1.1pt]{->}(0,0){2}{-110}{-75}
  \psarc[linewidth=1.1pt]{->}(0,0){2}{-75}{-45}
  \psarc[linewidth=1.1pt]{->}(0,0){2}{-45}{5}
  \psarc[linewidth=0.7pt,linestyle=dashed]{->}(0,0){2}{5}{110}
  \psarc[linewidth=1.1pt]{->}(0,0){2}{110}{140}
  \psarc[linewidth=0.7pt,linestyle=dashed]{->}(0,0){2}{140}{250}
  \rput[t](-0.087,-1.998){\footnotesize$q^{(\alpha,v)}_{\occ(\alpha,V)}$}
  \rput[tl](1,-1.732){\footnotesize$q^{(\alpha,v)}_{1}$}
  \rput[tl](1.879,-0.684){\footnotesize$q^{(\alpha,v)}_{2}$}
  \rput[b](-1.147,1.638){\footnotesize$q^{(\alpha,v)}_{r}$}
  \rput[br](-0.684,-1.879){\tiny$C^{(\alpha,v)}_{\occ(\alpha,V)}$}
  \rput[b](0.518,-1.932){\tiny$C^{(\alpha,v)}_{1}$}
  \rput[b](1.414,-1.414){\tiny$C^{(\alpha,v)}_{2}$}
  \rput[r](1.992,0.174){\tiny$C^{(\alpha,v)}_{3}$}
  \rput[tl](-0.684,1.879){\tiny$C^{(\alpha,v)}_{r}$}
  \rput[tl](-1.532,1.285){\tiny$C^{(\alpha,v)}_{r+1}$}
 \end{pspicture}
 \caption{}\label{fig4}
\end{figure}

For $\alpha\in\D_{\G}$ and $V\in\V_{(\alpha)}$ such that $\occ(\alpha,V)>1$, we define the \textit{non-special} $\alpha$-\textit{cycles at} $v$ to be the oriented cycles $q^{(\alpha,v)}_1,\ldots,q^{(\alpha,v)}_{\occ(\alpha,V)}$ in $\Q$. We denote these collection by $\neg\C^{\,v}_{(\alpha)}$. From the same construction we see that $|\C^{\,v}_{(\alpha)}|=|\neg\C^{\,v}_{(\alpha)}|=\occ(\alpha,V)$. When $\occ(\alpha,V)=1$ we say that $\neg\C^{\,v}_{(\alpha)}=\emptyset$. We have the following properties.\\

\begin{propo}\label{pr001}
 Let $\G$ be a Brauer configuration with induced quiver $\Q$ and $\alpha\in\D_{\G}$. $V\in\V_{(\alpha)}$ such that $\occ(\alpha,V)>1$, and $v$ the vertex in $\Q$ associated to the polygon $V$. Then reducing all the subindices modulo $\occ(\alpha,V)$ we have
 \begin{enumerate}
  \item\label{pr001.1} $C^{(\alpha,v)}_j=q^{(\alpha,v)}_j\cdots q^{(\alpha,v)}_{j+\occ(\alpha,V)-1},\,\,\forall1\le j\le\occ(\alpha,V).$
  \item\label{pr001.2} $\left(C^{(\alpha,v)}_j\right)^lq^{(\alpha,v)}_j=q^{(\alpha,v)}_j\left(C^{(\alpha,v)}_{j+1}\right)^l,\,\forall l\ge0;\,\forall1\le j\le\occ(\alpha,V).$
  \item\label{pr001.3} $\bigcup_{k=1}^{\occ(\alpha,V)-1}\left\{\,q^{(\alpha,v)}_l\cdots q^{(\alpha,v)}_{l+k-1}\,|\,1\le l\le\occ(\alpha,V)\,\right\}\subset v\Q v$, where $v\Q v$ is the collection of all the oriented cycles in $\Q$ with initial and final vertex $v$.
 \end{enumerate}
\end{propo}
\begin{proof}
 It follows using the Figure \ref{fig4}.
\end{proof}

Given a Brauer configuration $\G=(\G_0,\G_1,\mu,\o)$ we say that the polygon $V\in\G_1$ is a \textit{$d$-gon} if the number of vertices appearing in $V$ is $d$. We say that the configuration $\G$ is \textit{reduced} if and only if every polygon $V\in\G_1$ satisfies one of the following conditions:
\begin{enumerate}[(1)]
 \item $V\cap\T_{\G}=\emptyset$.
 \item If $V\cap\T_{\G}\neq\emptyset$, then $V$ is a $2$-gon with only one truncated vertex.
\end{enumerate}

Let $\L=K\Q/I$ be the Brauer configuration algebra associated to a reduced Brauer configuration $\G$. Denote by $\pi:K\Q\to\L$ the canonical surjection. When no confusion may arise we denote $\pi(x)$ by $\overline{x}$, for $x \in K\Q$. From the definition of relations of type one in $K\Q$ we see that for a $V\in\G_1$ and any $\alpha,\beta\in V\setminus \T_{\G}$\[\overline{C}^{\mu(\alpha)}=\overline{D}^{\mu(\beta)},\,\forall C\in\C^{\,v}_{(\alpha)},\,\forall D\in\C^{\,v}_{(\beta)}.\] So, let $\alpha$ be any nontruncated vertex in $V$, and let $C\in\C^{\,v}_{(\alpha)}$. We denote by $C^{(V)}$ any representative of the equivalence class \[\left\{C^{\mu(\alpha)}+x\,|\,x\in I_{\G}\right\}.\] Then, if $\beta\in V$ is another nontruncated vertex we have \begin{equation}\label{009}C^{(V)}=\overline{D}^{\mu(\beta)},\,\forall D\in\C_{(\beta)}^{\,v}.\end{equation} We observe that the definition of $C^{(V)}$ depends only on the polygon $V$.\\

Now we calculate an explicit $K$-basis for the vectorial space $v\L v$ where $v$ is the vertex in $\Q$ associated to the polygon $V\in\G_1$.

\begin{propo}\label{pr003}
 Let $\L=K\Q/I$ be the Brauer configuration algebra induced by $\G$. Then for any $V\in\G_1$ \[\textrm{dim}_Kv\L v=2+\sum_{\alpha\in \overline{V}}\occ(\alpha,V)(\occ(\alpha,V)\mu(\alpha)-1),\] where $\overline{V}=V\cap\G_0$. 
\end{propo}

\begin{proof}
 For the polygon $V\in\G_1$ define the following sets of vertices \begin{eqnarray*} \hat{V} & = & V\cap(\A_{\G}\cup\C_{\G}),\\ \A_{\G}^{V>1} & = & \left\{\,\alpha\in V\cap\A_{\G}\,|\,\occ(\alpha,V)>1\,\right\},\\ \B_{\G}^{V>1} & = & \left\{\,\beta\in V\cap\B_{\G}\,|\,\occ(\beta,V)>1\,\right\}.\end{eqnarray*} Now in the induced Brauer configuration algebra $\L$ define the following subsets of the $K$-space $v\L v${\footnotesize\begin{eqnarray*}\mathscr{L}_1 & = & \bigcup_{\alpha\in\hat{V}}\left\{\,\overline{C}^j\,|\,C\in\C^{\,v}_{(\alpha)},1\le j\le\mu(\alpha)-1\,\right\},\\ \mathscr{L}_2 & = & \bigcup_{\alpha\in\A_{\G}^{V>1}}\left(\bigcup_{k=1}^{\occ(\alpha,V)-1}\left\{\,\left(\overline{C^{(\alpha,v)}_l}\right)^j\overline{q^{(\alpha,v)}_l}\cdots\overline{q^{(\alpha,v)}_{l+k-1}}\,|\,0\le j\le\mu(\alpha)-1,1\le l\le\occ(\alpha,V)\,\right\}\right),\\ \mathscr{L}_3 & = & \bigcup_{\beta\in\B_{\G}^{V>1}}\left\{\,\overline{q^{(\beta,v)}_{l}}\cdots\overline{q^{(\beta,v)}_{l+k-1}}\,|\,1\le l\le\occ(\beta,V),1\le k\le\
occ(\beta,V)-1\,\right\}.\end{eqnarray*} } By using Proposition \ref{pr001} we have that the set $\mathscr{L}_1\cup\mathscr{L}_2\cup\mathscr{L}_3\cup\left\{\,\overline{v},C^{(V)}\,\right\}$ is a $K$-basis of $v\L v$. From the same definition of the collections $\mathscr{L}_i$'s, we see that {\small\begin{eqnarray}|\mathscr{L}_1| & = & \sum_{\alpha\in V\cap\A_{\G}}\occ(\alpha,V)(\mu(\alpha)-1)+\sum_{\gamma\in V\cap\C_{\G}}(\mu(\gamma)-1),\label{003}\\|\mathscr{L}_2| & = & \sum_{\alpha\in\A^{V>1}_{\G}}\occ(\alpha,V)(\occ(\alpha,V)-1)\mu(\alpha),\label{004}\\ |\mathscr{L}_3| & = & \sum_{\beta\in\B^{V>1}_{\G}}\occ(\beta,V)(\occ(\beta,V)-1)\label{005}\end{eqnarray}} but the expresions in (\ref{004}) and (\ref{005}) are respectively equal to {\small\begin{equation}\label{006}\sum_{\alpha\in V\cap\A_{\G}}\occ(\alpha,V)(\occ(\alpha,V)-1)\mu(\alpha),\end{equation}\begin{equation}\label{007}\sum_{\beta\in V\cap\B_{\G}}\occ(\beta,V)(\occ(\beta,V)-1).\end{equation}} Finally using the expressions (\ref{003}), (\ref{006}) 
and (\ref{007}) we obtain {\footnotesize\begin{multline*}\textrm{dim}_Kv\L v=2+\sum_{\alpha\in V\cap\A_{\G}}\occ(\alpha,V)(\mu(\alpha)-1)+\sum_{\gamma\in V\cap\C_{\G}}(\mu(\gamma)-1)\\+\sum_{\alpha\in V\cap\A_{\G}}\occ(\alpha,V)(\occ(\alpha,V)-1)\mu(\alpha)+\sum_{\beta\in V\cap\B_{\G}}\occ(\beta,V)(\occ(\beta,V)-1)\end{multline*}}{\scriptsize\begin{equation*}=2+\sum_{\alpha\in V\cap\A_{\G}}\occ(\alpha,V)(\occ(\alpha,V)\mu(\alpha)-1)+\sum_{\beta\in V\cap\B_{\G}}\occ(\beta,V)(\occ(\beta,V)-1)+\sum_{\gamma\in V\cap\C_{\G}}(\mu(\gamma)-1)\end{equation*}}{\footnotesize\begin{eqnarray} & = & 2+\sum_{\alpha\in V\cap(\A_{\G}\cup\B_{\G}\cup\C_{\G})}\occ(\alpha,V)(\occ(\alpha,V)\mu(\alpha)-1)\nonumber\\ & = &2+\sum_{\alpha\in V\cap(\G_0\setminus\T_{\G})}\occ(\alpha,V)(\occ(\alpha,V)\mu(\alpha)-1)\label{008}\end{eqnarray}} Now if $\alpha\in V$ is such that $\alpha\in\T_{\G}$ then $\val(\alpha)=\occ(\alpha,V)=1$ so $\occ(\alpha,V)\mu(\alpha)=1$, then the expression in (\ref{008}) remains equal to \[2+\sum_{\alpha\in V\cap\G_0}\occ(\alpha,V)(\occ(\alpha,V)\mu(\alpha)-1).\]
\end{proof}

\subsection{Induced Boundary Maps}\label{subsecc00b}
%Mostly part that we are going to mention at this subsection is taken from [ ] and [ ].

Let $\L$ be a finite dimensional $K$-algebra (not necessarily a Brauer configuration algebra), and $\L^e=\L^{\textrm{op}}\otimes_K\L$ the enveloping algebra of $\L$. Consider a $\L^e$-projective resolution of $\L$ \begin{equation}\label{p3}P\,:\,\cdots P_{n+1}\stackrel{d_{n+1}}{\longrightarrow}P_{n}\stackrel{d_n}{\longrightarrow}\cdots\stackrel{d_2}{\longrightarrow}P_1\stackrel{d_1}{\longrightarrow}P_0\stackrel{\pi}{\longrightarrow}\L\longrightarrow0\end{equation} Applying the functor $\hom_{\L^e}(-,\L)$ to $P$ in (\ref{p3}) we obtain the complex \[0\longrightarrow\hom_{\L^e}(P_0,\L)\stackrel{\widehat{d}_{1}}{\longrightarrow}\hom_{\L^e}(P_1,\L)\stackrel{\widehat{d}_{2}}{\longrightarrow}\cdots\stackrel{\widehat{d}_{n}}{\longrightarrow}\hom_{\L^e}(P_n,\L)\xrightarrow{\widehat{d}_{n+1}}\cdots,\] where $\hat{d}_{n}$ is de boundary map induced by $d_n$ for all $n\ge1$. From this complex  the $n$-th Hochschild cohomology group of $\L$ is defined by HH$^n(\L)=\textrm{ker}\hat{d}_{n+1}
/\textrm{im}\hat{d}_n$. We see that HH$^n(\L)=\textrm{Ext}_{\L^e}^n(\L,\L)$.\\

We are not interested, at the moment, in a complete $\L^e$-projective resolution of $\L$. In fact, we are interested in just a small segment of a particular projective resolution of $\L$, where $\L=K\Q/I$ with $\Q$ a finite quiver and $I$ an admissible ideal of $K\Q$. For this we use the construction in \cite{secnd}. Let's consider the following segment of the minimal $\L^e$-projective resolution of $\L=K\Q/I$ \begin{equation}\label{p4}P_1\stackrel{d_1}{\longrightarrow}P_0\stackrel{g}{\longrightarrow}\L\longrightarrow0,\end{equation} and where the $\L$-bimodules $P_0,P_1$ are given by \begin{eqnarray*}P_0 & = & \coprod_{v\in\Q_0}\L v\otimes v\L,\\ P_1 & = & \coprod_{a\in\Q_1}\L s(a)\otimes t(a)\L,\end{eqnarray*} and the $\L$-bimodule homomorphisms $g$ and $d_1$ are defined by
\[\begin{array}{ccc}P_0 & \stackrel{g}{\longrightarrow} & \L\\ v\otimes_v v & \longmapsto & v\end{array},\] \[\begin{array}{ccc}P_1 & \stackrel{d_1}{\longrightarrow} & P_0\\ s(a)\otimes_a t(a) & \longmapsto &  s(a)\otimes_{s(a)}\bar{a}-\bar{a}\otimes_{t(a)} t(a)\end{array}\] 

Here $\alpha_1\otimes_v\beta_1$ means the element in $P_0$ which has all its entries equal to zero except the $v$-entry which is equal to $\alpha_1\otimes\beta_1$. We have the same meaning for $P_1$, i.e, $\alpha_2\otimes_a\beta_2$ denotes the element in $P_1$ which has all its entries equal to zero except the $a$-entry which is equal to $\alpha_2\otimes\beta_2$.

Applying Hom$_{\L^e}(-,\L)$ to the complex in (\ref{p4}) we obtain  \begin{equation}\label{p5}0\longrightarrow\hom_{\L^e}(P_0,\L)\stackrel{\hat{d}_1}{\longrightarrow}\hom_{\L^e}(P_1,\L).\end{equation} It is well known that  HH$^0(\L)=\textrm{ker}\hat{d}_1=Z(\L)$, where $Z(\L)$ denotes the center of $\L$. We are not going to make the calculation of ker$\hat{d}_1$ directly from the homomorphism $\hat{d}_1$. We use the ideas in \cite{indbo} and \cite{nong} to do this.

We know that there exist natural isomorphisms of vector spaces \[\hom_{\L^e}(P_0,\L)\cong\coprod_{v\in\Q_0}v\L v,\]\[\hom_{\L^e}(P_1,\L)\cong\coprod_{a\in\Q_1}s(a)\L t(a).\] Let $v$ be a vertex  in the quiver $\Q$, $R_0=\coprod_{v\in\Q_0}v\L v$ and let $\pi'_v:R_0\to v\L v$ be the respective natural projection onto the component $v$. For $\chi\in R_0$ we denote by $\chi_v$ the element in $v\L v$ given by \begin{equation}\label{042}\chi_v:=\pi'_v(\chi).\end{equation} Now we give the morphism of vector spaces\[d_1^*:\coprod_{v\in\Q_0}v\L v\longrightarrow\coprod_{a\in\Q_1}s(a)\L t(a)\]
defined by \[d_1^*(\chi)=(\chi_{s(a)}\bar{a}-\bar{a}\chi_{t(a)})_{a\in\Q_1},\] for any $\chi\in\coprod_{v\in\Q_0}v\L v$.

It is not difficult to prove that the complex in (\ref{p5}) can be identified with the complex \begin{equation}\label{p6}0\longrightarrow\coprod_{v\in\Q_0}v\L v\stackrel{d_1^*}{\longrightarrow}\coprod_{a\in\Q_1}s(a)\L t(a),\end{equation} then it follows that $Z(\L)\cong\textrm{ker}d_1^*$. So, we will work with the morphism of $k$-spaces $d_1^*$, instead of $\hat{d}_1$, which is easier to handle.

\section{Dimension of the center of Brauer configuration algebras.}
\subsection{Characteristic elements.}\label{subsecc00d}

For each $\alpha\in\G_0\setminus\T_{\G}$ of the Brauer configuration $\G$ we denote by $C(\alpha)$ the element in the induced Brauer configuration algebra $\L$ given by \begin{equation}\label{041}C(\alpha):=\sum_{C\in\CC_{(\alpha)}}\overline{C}.\end{equation} Remember that for a polygon $V\in\G_1$ we denote by $C^{(V)}$ the element in $\L$ defined by the property in (\ref{009}).
\begin{propo}\label{pr002}
Let $\L$ be the algebra induced by the Brauer configuration $\G$. Then
\begin{enumerate}
 \item\label{pr002.1} $\left\{\,C^{(V)}\,|\,V\in\G_1\,\right\}\subset\textrm{soc}_{\L^{e}}\L$.
 \item\label{pr002.2} For $\alpha\in\D_{\G}$ let $C,D\in\CC_{(\alpha)}$ be special cycles.
 \begin{enumerate}
  \item\label{pr002.2.a} If $C\neq D$ then $\overline{C}\,\overline{D}=\overline{D}\,\overline{C}=0$.
  \item\label{pr002.2.b} If $C=D$ and $\alpha\in\B_{\G}$ then $\overline{C}^2=0$.
 \end{enumerate}
 \item\label{pr002.3} For any $\alpha\in\G_0\setminus\T_{\G}$, 
 \begin{enumerate}
 \item\label{pr002.5} $C(\alpha)^{\mu(\alpha)}=\sum\limits_{V\in\V_{(\alpha)}}\occ(\alpha,V)C^{(V)}$;
 \item\label{pr002.4} $\overline{C}^{\mu(\alpha)+1}=0$, for all $C\in\CC_{(\alpha)}$.
\end{enumerate}
\end{enumerate}
\end{propo}

\begin{proof}
 \ref{pr002.1}. It follows immediately from the definition of type two relations.\\
 
 \ref{pr002.2}. If a pair of two distinct elements from $\CC_{(\alpha)}$ are cycles starting at different vertices, then their respective product in $\L$ must be zero. Now if $C,D\in\C_{(\alpha)}^{\,v}$ for some $V\in\V_{(\alpha)}$, and $C\neq D$, then their product $CD$ must contain a relation of type three and $\overline{C}\,\overline{D}=0$ in $\L$. We have the same for the product $DC$. If $\alpha\in\B_{\G}$ and $C$ is any cycle in $\CC_{(\alpha)}$ we have that $C^2$ contais a relation of type two, i.e, $\overline{C}^2=0$ in $\L$.\\
 
 \ref{pr002.5}. We see that if $\alpha\in\C_{\G}$ the affirmation holds. Now, if $\alpha\in\D_{\G}$ then using what we proved in \ref{pr002.2} and the property in (\ref{009}) we have
 {\small\begin{eqnarray*}C(\alpha)^{\mu(\alpha)} & = & \sum_{C\in\CC_{(\alpha)}}\overline{C}^{\mu(\alpha)}\\ & = & \sum_{V\in\V_{(\alpha)}}\left(\sum_{C\in\C_{(\alpha)}^{\,v}}\overline{C}^{\mu(\alpha)}\right)\\ & = & \sum_{V\in\V_{(\alpha)}}\left(\sum_{C\in\C_{(\alpha)}^{\,v}}C^{(V)}\right)\\ & = & \sum_{V\in\V_{(\alpha)}}\left(\left(\sum_{C\in\C_{(\alpha)}^{\,v}}1_K\right)C^{(V)}\right)\\ & = & \sum_{V\in\V_{(\alpha)}}\occ(\alpha,V)C^{(V)}.\end{eqnarray*}
 }
 
 \ref{pr002.4}. It follows from the definition of type two relations.
\end{proof}
\begin{obs}\label{ob02}
Using the statement \ref{pr002.2} of Proposition \ref{pr002} we have that for any $\alpha\in\G_0\setminus\T_{\G}$ and any $j\ge1$ \begin{equation}\label{010}C(\alpha)^j=\left(\sum_{C\in\CC_{(\alpha)}}\overline{C}\right)^j=\sum_{C\in\CC_{(\alpha)}}\overline{C}^j.\end{equation}
\end{obs}

For $\alpha\in\G_0\setminus\T_{\G}$ let $\olCC_{(\alpha)}$ denote the subset of $\Q_1$ defined by \begin{equation}\label{2.0013}\olCC_{(\alpha)}:=\left\{a\in\Q_1\,|a\textrm{ is contained in a special }\alpha\textrm{-cycle}\right\}.\end{equation} We call the collection $\olCC_{(\alpha)}$ the \textit{set of arrows associated to} $\alpha$. Let $h:\D_{\G}\to\left\{0,1\right\}$ be the map defined by \begin{equation}\label{2.0018}h(\alpha)=\left\{\begin{array}{cl}0, &\textrm{ if }\olCC_{(\alpha)}\textrm{ contains no loops};\\ 1, &\textrm{ if }\olCC_{(\alpha)}\textrm{ contains loops}.\end{array}\right.\end{equation} It is easy to check that \begin{equation}\label{2.0019}\alpha\in h^{-1}(1)\iff\begin{array}{l}\exists V\in\V_{(\alpha)}\,;\,\occ(\alpha,V)>1\textrm{ and}\\\exists1\le s\le\occ(\alpha,V)\,;\, q^{(\alpha,v)}_s\in\Q_1.\end{array}\end{equation} For $\alpha\in h^{-1}(1)$ let $\M^{(\alpha)}$ be the set of ordered pairs defined by \begin{equation}\label{2.0020}\M^{(\alpha)}=\left\{(V,s)\,|\,V\textrm{ and }s\textrm{ satisfies }(\ref{2.0019})\right\}.\end{equation} Each pair $(V,s)\in\M^{(\alpha)}$ is associated to a unique loop in $\Q$, the non-special $\alpha$-cycle $q^{(\alpha,v)}_s$ in (\ref{2.0019}). For $(V,s)\in\M^{(\alpha)}$ let $D^{(\alpha)}_{V,s}$ be the cycle in $\Q$ defined by \begin{equation}\label{2.0021}D^{(\alpha)}_{V,s}=\left\{\begin{array}{cl}\left(C^{(\alpha,v)}_{s+1}\right)^{\mu(\alpha)-1}q^{(\alpha,v)}_{s+1}\cdots q^{(\alpha,v)}
_{s-1}, & \alpha\in\A_{\G};\\q^{(\alpha,v)}_{s+1}\cdots q^{(\alpha,v)}_{s-
1}, & \alpha\in\B_{\G}.\end{array}\right.\end{equation} In (\ref{2.0021}) the composition of cycles $q^{(\alpha,v)}_{s+1}\cdots q^{(\alpha,v)}_{s-1}$ means \[q^{(\alpha,v)}_{s+1}\cdots q^{(\alpha,v)}_{\occ(\alpha,V)}q^{(\alpha,v)}_{1}\cdots q^{(\alpha,v)}_{s-1}.\] We call $D^{(\alpha)}_{V,s}$ a \textit{central mixed} $\alpha$\textit{-cycle at the vertex} $v$ in $\Q$. We denote the collection of these cycles by $\Psi^{\,\,\,v}_{(\alpha)}$, i.e, \begin{equation}\label{2.0022}\Psi^{\,\,\,v}_{(\alpha)}=\left\{D^{(\alpha)}_{V,s}\,|\,(V,s)\in\M^{(\alpha)}\right\}.\end{equation} We recall that a path $p$ is a \textit{prefix of} the path $q$ in the quiver $\Q$ if the first arrow of $p$ coincides with the first arrow of $q$ and $p$ is a subpath of $q$. The following lemma is very useful.

\begin{lem}\label{2.lem002}
 Let $\L$ be a Brauer configuration algebra associated to the Brauer configuration $\G$. Let $x$ be an element in $\L$ that satisfies the following property\begin{quote}$\forall\alpha\in\G_0\setminus\T_{\G},\forall C\in\CC_{(\alpha)}$\\\[p\textrm{ is a prefix of }C^{\mu(\alpha)}\implies\overline{p}x=x\overline{p}.\]\end{quote} Then $x\in Z(\L)$.
\end{lem}
\begin{proof}
 This follows from \cite[Proposition 3.3]{brau}.
\end{proof}

Let $\G=(\G_0,\G_1,\mu,\o)$ be a reduced Brauer configuration with induced quiver $\Q$ and associated Brauer configuration algebra $\L=K\Q/I$. For $\alpha\in\C_{\G}$ let $\beta$ be a nontruncated vertex such that $\beta\neq\alpha$ and let $p$ be a prefix of $C^{\mu(\beta)}$, for some $C\in\CC_{(\beta)}$. Let $a^{(\alpha)}$ be the unique element in $\olCC_{(\alpha)}$, then $a^{(\alpha)}$ is a loop. If $pa^{(\alpha)}\neq0$ and $a^{(\alpha)}p=0$ in $K\Q$ then $pa^{(\alpha)}$ contains a relation of type three and hence $\overline{pa^{(\alpha)}}=\overline{a^{(\alpha)}p}=0$.  In the same way, if $a^{(\alpha)}p\neq0$ but $pa^{(\alpha)}=0$ then $a^{(\alpha)}p$ also contains a relation of type three and hence $\overline{a^{(\alpha)}p}=\overline{pa^{(\alpha)}}=0.$ Now, if $a^{(\alpha)}p\neq0$ and $pa^{(\alpha)}\neq0$ then $p$ must be an oriented cycle and hence both $a^{(\alpha)}p$ and $pa^{(\alpha)}$ contain a relation of type three, and also in this case $\overline{a^{(\alpha)}p}=\overline{pa^{(\alpha)}}=0$. Now, if 
$\beta=\alpha$ then $p$ must be a prefix of $\left(a^{(\alpha)}\right)^{\mu(\alpha)}$ and obviously $p$ commutes with $a^{(\alpha)}$.

\begin{propo}\label{2.pr005}
 Let $\G=(\G_0,\G_1,\mu,\o)$ be a reduced Brauer configuration with associated Brauer configuration algebra $\L$. Then \begin{enumerate}
 \item\label{2.pr005.3} \[\left\{C^{(V)}\,|\,V\in\G_1\right\}\subset Z(\L);\]
 \item\label{2.pr005.1} \[\bigcup_{\alpha\in\C_{\G}}\left\{\bar{a}\,|\,a\in\olCC_{(\alpha)}\right\}\subset Z(\L);\]
 \item\label{2.pr005.2} \[\bigcup_{\alpha\in h^{-1}(1)}\left\{\,\overline{D^{(\alpha)}_{V,s}}\big{|}\,(V,s)\in\M^{(\alpha)}\right\}\subset Z(\L).\]
 \end{enumerate}
\end{propo}

\begin{proof}
 \ref{2.pr005.3}. It is obvious by Proposition \ref{pr002}.\ref{pr002.1}.\\
 
 \ref{2.pr005.1}. It follows from the previous observations and Lemma \ref{2.lem002}.\\
 
 \ref{2.pr005.2}. Let $\alpha\in h^{-1}(1)$ and $(V,s)\in\M^{(\alpha)}$ be an ordered pair that satisfies (\ref{2.0019}). So, we have that the non-special cycle $q^{(\alpha,v)}_s$ is a loop and the elements $C^{(\alpha,v)}_{s},C^{(\alpha,v)}_{s+1}$ and $q^{(\alpha,v)}_{s}$ can be graphically represented as
\begin{figure}[H]
\centering
%\psset{linestyle=none}
 \begin{pspicture}(-2.8,-2.8)(2.8,2.8)%\grilla
  \psarc[linewidth=1.1pt]{->}(0,0){2}{-105}{-75}%C1
  \psarc[linewidth=1.1pt]{->}(0,0){2}{-75}{-45}%C2
  \psarc[linewidth=1.1pt]{->}(0,0){2}{-45}{-15}%C3
  \psarc[linewidth=0.7pt,linestyle=dashed]{->}(0,0){2}{-15}{255}
  \rput[t](-0.087,-1.998){\footnotesize$q^{(\alpha,v)}_{s-1}$}
  \rput[tl](1,-1.732){\footnotesize$q^{(\alpha,v)}_{s}$}
  \rput[tl](1.732,-1){\footnotesize$q^{(\alpha,v)}_{s+1}$}
  %\rput[b](-1.147,1.638){\footnotesize$q^{(\alpha,v)}_{r}$}
  %\rput[br](-0.684,-1.879){\tiny$C^{(\alpha,v)}_{\occ(\alpha,V)}$}
  \rput[b](0.492,-1.8325){\footnotesize$C^{(\alpha,v)}_{s}$}
  \rput[b](1.273,-1.273){\footnotesize$C^{(\alpha,v)}_{s+1}$}
  %\rput[r](1.992,0.174){\tiny$C^{(\alpha,v)}_{3}$}
  %\rput[tl](-0.684,1.879){\tiny$C^{(\alpha,v)}_{r}$}
  %\rput[tl](-1.532,1.285){\tiny$C^{(\alpha,v)}_{r+1}$}
  %\rput[t](0,-2.7){\figu\label{2.001}}
  \end{pspicture}
  \caption{}\label{fig6}
 \end{figure}
 Let $D$ be the element in $\Psi^{\,\,\,v}_{(\alpha)}$ associated to the pair $(V,s)$. Let $\beta\in\G_0\setminus\T_{\G}$ and $p$ be a prefix of $C^{\mu(\beta)}$ for some $C\in\CC_{(\beta)}$. Suppose that $\alpha\neq\beta$. If $Dp\neq0$ and $pD=0$ in $K\Q$ then the path $Dp$ contains a relation of type three, and hence $\overline{Dp}=\overline{Dp}=0$ in $\L$. In a similar way we have that if $Dp=0$ but $pD\neq0$ in $K\Q$, then $\overline{pD}=\overline{Dp}=0$ in $\L$. If $pD\neq0$ and $Dp\neq0$ then $p$ must be a cycle, but also in this case both $pD$ and $Dp$ contain a relation of type thre, and hence $\overline{pD}=\overline{Dp}=0$. Suppose now that $\alpha=\beta$, and without loss of generality assume that $\alpha\in\A_{\G}$ (in the case $\alpha\in\B_{\G}$ the reasonig is analogous). It is sufficient to consider only the case when $p$ is a prefix of {\footnotesize$\left(C^{(\alpha,v)}_{s}\right)^{\mu(\alpha)}$}. If $\ell(p)>1$ then is easy to see that $Dp$ contains a relation of type two, and if $pD\neq0$ 
in $K\Q$ then also contains a relation of type two. Hence $\overline{Dp}=\overline{pD}=0$. If $\ell(p)=1$ then necessarily $p=q^{(\alpha,v)}_s$, and by Proposition \ref{pr001} we obtain \begin{eqnarray}\overline{D}\overline{p} & = & \left(\overline{C^{(\alpha,v)}_{s+1}}\right)^{\mu(\alpha)-1}q^{\overline{(\alpha,v)}}_{s+1}\cdots q^{\overline{(\alpha,v)}}_{s-1}q^{\overline{(\alpha,v)}}_{s}\nonumber\\ & = & \left(\overline{C^{(\alpha,v)}_{s+1}}\right)^{\mu(\alpha)}\label{2.0023}\\\overline{p}\overline{D} & = & q^{\overline{(\alpha,v)}}_{s}\left(\overline{C^{(\alpha,v)}_{s+1}}\right)^{\mu(\alpha)-1}q^{\overline{(\alpha,v)}}_{s+1}\cdots q^{\overline{(\alpha,v)}}_{s-1}\nonumber\\ & = & \left(\overline{C^{(\alpha,v)}_{s}}\right)^{\mu(\alpha)}.\label{2.0024}\end{eqnarray} Note that {\footnotesize$\left(C^{(\alpha,v)}_{s}\right)^{\mu(\alpha)}-\left(C^{(\alpha,v)}_{s+1}\right)^{\mu(\alpha)}$} is a 
relation of type three, then (\ref{2.0023}) and (\ref{2.0024}) are equal. By Lemma \ref{2.lem002} we can conclude that $\overline{D}$ is in $Z(\L)$.
\end{proof}

\subsection{Calculating the dimension of the center.}
We start this section with two lemmas. 

\begin{lem}\label{pr004}
Let $\L$ be the Brauer configuration algebra associated to the reduced and connected Brauer configuration $\G$. For each $a\in\Q_1$ there exists a unique $\alpha\in\G_0\setminus\T_{\G}$ such that
\begin{enumerate}
 \item\label{pr004.1} If $\alpha\notin\B_{\G}$ then there exist a unique pair $(C,C')$ of elements in $\CC_{(\alpha)}$ such that $\overline{Ca}\neq0$ and $\overline{aC'}\neq0$.
 \item\label{pr004.2} If $\alpha\in\B_{\G}$ then $\overline{Da}=\overline{aE}=0$, for any $D,E\in\CC$, where $\CC$ is the collection defined in (\ref{040}).
\end{enumerate}
\end{lem}

\begin{proof}
 \ref{pr004.1}. By Proposition \ref{pr005} there exists a unique $\alpha\in\G_0\setminus\T_{\G}$  such that $a$ occurs once in any of the special cycles in $\CC_{(\alpha)}$. Let $(C,C')$ be the pair of elements of $\CC_{(\alpha)}$ defined by
 \begin{itemize}
  \item $a$ is the first arrow of $C$;
  \item $a$ is the last arrow of $C'$.
 \end{itemize}
Then $(C,C')$ is the required pair.\\
 
\ref{pr004.2}. It follows from the type two and type three relations.
\end{proof}
If $\alpha\in\A_{\G}$ it follows from the statement of Lemma \ref{pr004}.\ref{pr004.1} that, without lost of generality, we can assume that $a=a^{(\alpha)}_i$, where $1\le i\le\val(\alpha)$, as in the sequence in (\ref{001}). Then the special cycles $C$ and $C'$ are given respectively by \begin{equation}\label{013}\begin{array}{rcl}C & = & a^{(\alpha)}_{i}\cdots a^{(\alpha)}_{i-1},\\ C' & = & a^{(\alpha)}_{i+1}\cdots a^{(\alpha)}_{i},\end{array}\end{equation} where $a^{(\alpha)}_1,\ldots,a^{(\alpha)}_{\val(\alpha)}$ is the collection of arrows induced by the sucessor sequence of the vertex $\alpha$. Now, if the vertex $\alpha$ belongs to $\C_{\G}$ necessarily the arrow $a$ is a loop, and so $a=C=C'$. We have the following lemma.
\begin{lem}\label{le001}
 Assuming all the hypothesis of Lemma \ref{pr004}, for $a\in\Q_1$  let $\alpha\in\G_0\setminus\T_{\G}$ be the unique vertex such that $\alpha\notin\B_{\G}$ and let $C,C'\in\CC_{(\alpha)}$ be the unique special cycles associated to $\alpha$ and $a$ as set up in (\ref{013}). For any positive integer $l$ we have that \[C^la=aC'^l.\]
\end{lem}
\begin{proof}
The proof is straightforward and is left to the reader.\\
\end{proof}

Let $\L=K\Q/I$ be the Brauer configuration algebra associated to the reduced Brauer configuration $\G$, and assume that $\L$ is indecomposable and rad$^2(\L)\neq0$. Let $\chi$ be an element of the $K$-space $\coprod_{v\in\Q_0}v\L v$. For a vertex $v$ in $\Q$ let $\chi_v$ denote the element defined in (\ref{042}). By Proposition \ref{pr003} and its proof we have that the element $\chi_v$ of $v\L v$ can be expressed by
{\footnotesize\begin{multline}\label{011}\chi_v=x^{(v)}v+\sum_{\alpha\in\widehat{V}}\left(\sum_{j=1}^{\mu(\alpha)-1}\left(\sum_{C\in\C^{\,v}_{(\alpha)}}y^{(\alpha)}_{j,f(C)}\overline{C}^j\right)\right)+z^{(v)}C^{(V)}\\+\sum_{\alpha\in\A_{\G}^{V>1}}\left(\sum_{k=1}^{\occ(\alpha,V)-1}\left(\sum_{l=1}^{\occ(\alpha,V)}\left(\sum_{j=0}^{\mu(\alpha)-1}y^{(\alpha)}_{j,f(C^{(\alpha,v)}_l),k}\left(\overline{C^{(\alpha,v)}_{l}}\right)^j\overline{q^{(\alpha,v)}_{l}}\cdots \overline{q^{(\alpha,v)}_{l+k-1}}\right)\right)\right)\\+\sum_{\beta\in\B_{\G}^{V>1}}\left(\sum_{k=1}^{\occ(\beta,V)-1}\left(\sum_{l=1}^{\occ(\beta,V)}y^{(\beta)}_{g(q^{(\beta,v)}_l),k}\overline{q^{(\beta,v)}_{l}}\cdots \overline{q^{(\beta,v)}_{l+k-1}}\right)\right)\end{multline}} where $x^{(v)},z^{(v)},y^{(\alpha)}_{j,f(C)},y^{(\alpha)}_{j,f(C^{(\alpha,v)}_l),k},y^{(\beta)}_{g(q^{(\beta,v)}_l),k}$ are scalars in the field $K$, and and $g:\neg\CC\to\Q_1$ the map that sends a non-special cycle to its first arrow, with \[\neg\CC_{(\alpha)}=\bigcup_{V\in\V_{(\alpha)}}\neg\C^{\,v}_{(\alpha)}\] for every $\alpha\in\D_{\G}$, and \[\neg\CC=\bigcup_{\alpha\in\D_{\G}}\neg\CC_{(\alpha)}.\] If $d^*_1:\coprod_{v\in\Q_0}v\L v\to\coprod_{a\in\Q_1} s(a)\L t(a)$ is the $\L$-bimodule homomorphism defined in Subsection \ref{subsecc00b} we have that the image of $\chi$ under $d_1^*$ can be computed as \begin{equation}\label{014}d_1^*(\chi)=\left(\chi_{s(a)}\bar{a}-\bar{a}\chi_{t(a)}\right)_{a\in\Q_1}.\end{equation} Let $a$ be an arrow of the quiver $\Q$, and let $\alpha$ be the unique vertex in Lemma \ref{pr004} associated to the arrow $a$. For the vertex $\alpha$ we have the following possible cases.
\begin{enumerate}[(a)]
 \item\label{1.a} $\alpha\in\A_{\G}$;
 \item\label{1.b} $\alpha\in\B_{\G}$;
 \item\label{1.c} $\alpha\in\C_{\G}$.
\end{enumerate}

For the arrow $a$ let $V,W$ be polygons of the configuration $\G$ associated to the vertices $v,w$ of the quiver $\Q$ respectively, and such that $s(a)=v$ and $t(a)=w$. If $\alpha\in\D_{\G}$ we can suppose that $a=a^{(\alpha)}_i$ for some $1\le i\le\val(\alpha)$, and the unique cycles $C,C'$ of Lemma \ref{pr004} are given respectively by \begin{eqnarray*}C^{(\alpha,v)}_s & = & a^{(\alpha)}_i\cdots a^{(\alpha)}_{i-1},\\ C^{(\alpha,w)}_r & = & a^{(\alpha)}_{i+1}\cdots a^{(\alpha)}_i,\end{eqnarray*} where $1\le s\le\occ(\alpha,V)$ and $1\le r\le\occ(\alpha,W)$. For the understanding of some further computations we will represent graphically the location of $C^{(\alpha,v)}_s$ and $C^{(\alpha,w)}_r$ by 

\begin{figure}[H]
\centering
 \begin{pspicture}(-2.8,-2.8)(2.8,2.8)%\grilla
  \psarc[linewidth=1.1pt]{->}(0,0){2}{-120}{-90}
  \psarc[linewidth=1.1pt]{->}(0,0){2}{-90}{-60}
  \psarc[linewidth=1.1pt]{->}(0,0){2}{-60}{-30}
  \psarc[linewidth=0.7pt,linestyle=dashed]{->}(0,0){2}{-30}{240}
  \rput[tr](-0.518,-1.932){\tiny$a^{(\alpha)}_{i-1}$}
  \rput[t](0.518,-1.932){\tiny$a^{(\alpha)}_{i}$}
  \rput[tl](1.414,-1.414){\tiny$a^{(\alpha)}_{i+1}$}
  \rput[b](0,-1.9){\tiny$C^{(\alpha,v)}_s$}
  \rput[b](0.95,-1.645){\tiny$C^{(\alpha,w)}_r$}
 \end{pspicture}
 \caption{}\label{fig7}
\end{figure}

\noindent (\ref{1.a}) Case $\alpha\in\A_{\G}$: Using expression (\ref{011}) for $\chi_v$ and $\chi_w$ and Lemma \ref{pr004}, we have that the $a$-entry in (\ref{014}) is equal to
{\scriptsize\begin{multline}\label{015}\left(x^{(v)}-x^{(w)}\right)\bar{a}+\sum_{j=1}^{\mu(\alpha)-1}y^{(\alpha)}_{j,a^{(\alpha)}_i}\left(\overline{C^{(\alpha,v)}_s}\right)^j\bar{a}\\+\sum_{l=1}^{\occ(\alpha,V)-1}\left(\sum_{j=0}^{\mu(\alpha)-1}y^{(\alpha)}_{j,f(C^{(\alpha,v)}_{s+l}),\occ(\alpha,V)-l}\left(\overline{C^{(\alpha,v)}_{s+l}}\right)^jq^{\overline{(\alpha,v)}}_{s+l}\cdots q^{\overline{(\alpha,v)}}_{s+\occ(\alpha,V)-1}\bar{a}\right)\\-\left(\sum_{k=1}^{\occ(\alpha,W)-1}\left(\sum_{j=0}^{\mu(\alpha)-1}y^{(\alpha)}_{j,a^{(\alpha)}_{i+1},k}\bar{a}\left(\overline{C^{(\alpha,w)}_{r}}\right)^jq^{\overline{(\alpha,w)}}_{r}\cdots q^{\overline{(\alpha,w)}}_{r+k-1}\right)+\sum_{j=1}^{\mu(\alpha)-1}y^{(\alpha)}_{j,a^{(\alpha)}_{i+1}}\bar{a}\left(\overline{C^{(\alpha,w)}_r}\right)^j\right)
\end{multline}}
By Proposition \ref{pr001}.\ref{pr001.2} and Lemma \ref{le001} we have {\footnotesize\begin{eqnarray*}\left(C^{(\alpha,v)}_{s+l}\right)^jq^{(\alpha,v)}_{s+l}\cdots q^{(\alpha,v)}_{s+\occ(\alpha,V)-1}a & = & q^{(\alpha,v)}_{s+l}\cdots q^{(\alpha,v)}_{s+\occ(\alpha,V)-1}\left(C^{(\alpha,v)}_s\right)^ja,\\ a\left(C^{(\alpha,w)}_r\right)^j & = & \left(C^{(\alpha,v)}_s\right)^ja\end{eqnarray*}}for all $0\le j\le\mu(\alpha)-1$, and all $1\le l\le\occ(\alpha,V)-1$. So applying these equalities in (\ref{015}) and then collecting similar terms we obtain {\scriptsize\begin{multline}\label{017}\left(x^{(v)}-x^{(w)}\right)\bar{a}+\sum_{l=1}^{\occ(\alpha,V)-1}\left(\sum_{j=0}^{\mu(\alpha)-1}y^{(\alpha)}_{j,f(C^{(\alpha,v)}_{s+l}),\occ(\alpha,V)-l}q^{\overline{(\alpha,v)}}_{s+l}\cdots q^{\overline{(\alpha,v)}}_{s+\occ(\alpha,V)-1}\left(\overline{C^{(\alpha,v)}_{s}}\right)^j\bar{a}\right)\\+\sum_{j=1}^{\mu(\alpha)-1}\left(y^{(\alpha)}_{j,a^{(\alpha)}_i}-y^{(\alpha)}_{j,a^{(\alpha)}_{i+1}}\right)\left(\overline{C^{(\alpha,v)
}_s}\right)^j\bar{a}-\sum_{k=1}^{\occ(\alpha,W)-1}\left(\sum_{j=0}^{\mu(\alpha)-1}y^{(\alpha)}_{j,a^{(\alpha)}_{i+1},k}\left(\overline{C^{(\alpha,v)}_s}\right)^j\bar{a}q^{\overline{(\alpha,w)}}_{r}\cdots q^{\overline{(\alpha,w)}}_{r+k-1}\right)\end{multline}}Now, if $a=a^{(\alpha)}_i$ is a loop then we obtain that $q^{(\alpha,v)}_s=a$, $v=w$, $r=s+1$, and by Figure \ref{fig7}, $f(C^{(\alpha,v)}_{s+1})=a^{(\alpha)}_{i+1}$. For the particular values $l=1$ and $k=\occ(\alpha,V)-1$ we can regroup terms in (\ref{017}) which have same scalars. Then in (\ref{017}) appears the expression {\footnotesize\begin{multline}\label{043}\sum_{j=0}^{\mu(\alpha)-1}y^{(\alpha)}_{j,a^{(\alpha)}_{i+1},\occ(\alpha,V)-1}\left(\left(\overline{C^{(\alpha,v)}_{s+1}}\right)^jq^{\overline{(\alpha,v)}}_{s+1}\cdots q^{\overline{(\alpha,v)}}_{s-1}\bar{a}-\left(\overline{C^{(\alpha,v)}_{s}}\right)^j\bar{a}q^{\overline{(\alpha,v)}}_{s+1}\cdots q^{\overline{(\alpha,v)}}_{s-1}\right)\\=\sum_{j=0}^{\mu(\alpha)-2}y^{(\alpha)}_{j,a^{(\alpha)}_{
i+1},\occ(\alpha,V)-1}\left(\left(\overline{C^{(\alpha,v)}_{s+1}}\right)^{j+1}-\left(\overline{C^{(\alpha,v)}_{s}}\right)^{j+1}\right). \end{multline}} The right hand side of the equality in (\ref{043}) holds because \begin{itemize}\item {\small$\overline{C^{(\alpha,v)}_{s+1}}=q^{\overline{(\alpha,v)}}_{s+1}\cdots q^{\overline{(\alpha,v)}}_{s-1}\bar{a}$} and {\small$\overline{C^{(\alpha,v)}_{s}}=\bar{a}q^{\overline{(\alpha,v)}}_{s+1}\cdots q^{\overline{(\alpha,v)}}_{s-1}$} by Proposition \ref{pr001}.\ref{pr001.1};
\item {\footnotesize$\left(\overline{C^{(\alpha,v)}_{s+1}}\right)^{\mu(\alpha)}=\left(\overline{C^{(\alpha,v)}_{s}}\right)^{\mu(\alpha)}$} by type one relations.
\end{itemize} In conclusion, when $a=a^{(\alpha)}_i$ is a loop (\ref{017}) is equal to
{\scriptsize\begin{multline}\label{044}\sum_{j=1}^{\mu(\alpha)-1}\left(y^{(\alpha)}_{j,a^{(\alpha)}_{i}}-y^{(\alpha)}_{j,a^{(\alpha)}_{i+1}}\right)\left(\overline{C^{(\alpha,v)}_s}\right)^j\bar{a}\\+\sum_{\substack{1\le l\le\occ(\alpha,V)-1\\0\le j\le\mu(\alpha)-1\\l\neq1}}y^{(\alpha)}_{j,f\left(C^{(\alpha,v)}_{s+l}\right),\occ(\alpha,V)-l}q^{\overline{(\alpha,v)}}_{s+l}\cdots q^{\overline{(\alpha,v)}}_{s-1}\left(\overline{C^{(\alpha,v)}_s}\right)^j\bar{a}\\-\sum_{\substack{1\le k\le\occ(\alpha,V)-1\\0\le j\le\mu(\alpha)-1\\k\neq\occ(\alpha,V)-1}}y^{(\alpha)}_{j,a^{(\alpha)}_{i+1},k}\left(\overline{C^{(\alpha,v)}_s}\right)^j\bar{a}q^{\overline{(\alpha,v)}}_{s+1}\cdots q^{\overline{(\alpha,v)}}_{s+k}\\+\sum_{j=0}^{\mu(\alpha)-2}y^{(\alpha)}_{j,a^{(\alpha)}_{i+1},\occ(\alpha,V)-1}\left(\left(\overline{C^{(\alpha,v)}_{s+1}}\right)^{j+1}-\left(\overline{C^{(\alpha,v)}_{s}}\right)^{j+1}\right)\end{multline}}Note that the expression obtained in (\ref{044}) was expected. One of the missing scalars in (\ref{044}) 
is exactly the same associated to the class of the elements defined in (\ref{2.0021}) of Subsection \ref{subsecc00d}, which is associated to the loop $q^{(\alpha,v)}_s=a^{(\alpha)}_i$ (the element associated to the scalar {\scriptsize$y^{(\alpha)}_{\mu(\alpha)-1,a^{(\alpha)}_{i+1},\occ(\alpha,V)-1}$}). We showed in Proposition \ref{2.pr005}.\ref{2.pr005.2} that this type of elements are in the center of the algebra.\\

\noindent (\ref{1.b}) Case $\alpha\in\B_{\G}$: Using again (\ref{011}) to compute $\chi_v$ and $\chi_w$ and Lemma \ref{pr004}, we have that the $a$-entry  in (\ref{014}), in this case, is equal to
{\footnotesize\begin{multline}\label{016}\left(x^{(v)}-x^{(w)}\right)\bar{a}+\sum_{l=1}^{\occ(\alpha,V)-1}y^{(\alpha)}_{g(q^{(\alpha,v)}_{s+l}),\occ(\alpha,V)-l}q^{\overline{(\alpha,v)}}_{s+l}\cdots q^{\overline{(\alpha,v)}}_{s+\occ(\alpha,V)-1}\bar{a}\\-\sum_{k=1}^{\occ(\alpha,W)-1}y^{(\alpha)}_{a^{(\alpha)}_{i+1},k}\bar{a}q^{\overline{(\alpha,w)}}_{r}\cdots q^{\overline{(\alpha,w)}}_{r+k-1}\end{multline}}  Also in this case, if $a=a^{(\alpha)}_i$ is a loop then $q^{(\alpha,v)}_s=a$, $v=w$, $r=s+1$ and $g(q^{(\alpha,v)}_{s+1})=a^{(\alpha)}_{i+1}$. For the particular values $l=1$ and $k=\occ(\alpha,V)-1$ we have that in (\ref{016}) appears the expression
{\footnotesize\begin{multline}y^{(\alpha)}_{a^{(\alpha)}_{i+1},\occ(\alpha,V)-1}q^{\overline{(\alpha,v)}}_{s+1}\cdots q^{\overline{(\alpha,v)}}_{s-1}\overline{a}-y^{(\alpha)}_{a^{(\alpha)}_{i+1},\occ(\alpha,V)-1}\overline{a}q^{\overline{(\alpha,v)}}_{s+1}\cdots q^{\overline{(\alpha,v)}}_{s-1}\\=y^{(\alpha)}_{a^{(\alpha)}_{i+1},\occ(\alpha,V)-1}q^{\overline{(\alpha,v)}}
_{s+1}\cdots q^{\overline{(\alpha,v)}}_{s}-y^{(\alpha)}_{a^{(\alpha)}_{i+1},\occ(\alpha,V)-1}q^{\overline{(\alpha,v)}}_{s}\cdots q^{\overline{(\alpha,v)}}_{s-1}\\=y^{(\alpha)}_{a^{(\alpha)}_{i+1},\occ(\alpha,V)-1}\overline{C^{(\alpha,v)}_{s+1}}-y^{(\alpha)}_{a^{(\alpha)}_{i+1},\occ(\alpha,V)-1}\overline{C^{(\alpha,v)}_s}=0.\end{multline}} This is due the fact that $\overline{C^{(\alpha,v)}_{s+1}}=\overline{C^{(\alpha,v)}_{s}}$ and Proposition \ref{pr001}.\ref{pr001.1}. So, when $a=a^{(\alpha)}_i$ is a loop then (\ref{016}) is equal to {\footnotesize\begin{multline}\label{2.0038}\sum_{l=2}^{\occ(\alpha,V)-1}y^{(\alpha)}_{g\left(q^{(\alpha,v)}_{s+l}\right),\occ(\alpha,V)-l}q^{\overline{(\alpha,v)}}_{s+l}\cdots q^{\overline{(\alpha,v)}}_{s}-\sum_{k=1}^{\occ(\alpha,V)-2}y^{(\alpha)}_{a^{(\alpha)}_{i+1},k}q^{\overline{(\alpha,v)}}_{s}\cdots q^{\overline{(\alpha,v)}}_{s+k}\end{multline}}As in the case (\ref{1.a}), the expression obtained in (\ref{2.0038}) was expected. The missing scalar, {\scriptsize$y^{(\alpha)}_
{a^{(\alpha)}_{i+1},\occ(\alpha,V)-1}$}, is associated to the class of elements defined in (\ref{2.0021}), which belong to the center of the algebra.\\

\noindent(\ref{1.c}) Case $\alpha\in\C_{\G}$: In this case the arrow $a$ is necessarily a loop, and by Proposition \ref{2.pr005}.\ref{2.pr005.1} the $a$-entry in (\ref{014}) is equal to zero.

\begin{obs}\label{ob01}
 We give some remarks about the appearing terms in expressions (\ref{017}), (\ref{044}), (\ref{016}) and (\ref{2.0038}). When the arrow $a$ is not a loop, the collection 
{\scriptsize
\begin{multline}
\biggl\{\left(\overline{C^{(\alpha,v)}_s}\right)^j\bar{a},q^{\overline{(\alpha,v)}}_{s+l}\cdots q^{\overline{(\alpha,v)}}_{s+\occ(\alpha,V)-1}\left(\overline{C^{(\alpha,v)}_{s}}\right)^j\bar{a},\left(\overline{C^{(\alpha,v)}_s}\right)^j\bar{a}q^{\overline{(\alpha,w)}}_{r}\cdots q^{\overline{(\alpha,w)}}_{r+k-1}\,\Big{|}\\ 0\le  j\le\mu(\alpha)-1,1\le l\le\occ(\alpha,V)-1,1\le k\le\occ(\alpha,W)-1\biggr\}
\end{multline}}
is a linearly independent subset of the $K$-space $\L$. This is due the fact that the collection {\footnotesize\begin{equation}\label{2.0039}\left\{\,q^{\overline{(\alpha,v)}}_{s+l}\cdots q^{\overline{(\alpha,v)}}_{s-1}\left(\overline{C^{(\alpha,v)}_s}\right)^j\bar{a}\,\Big{|}\,0\le j\le\mu(\alpha)-1,1\le l\le\occ(\alpha,V)-1\,\right\}\end{equation}}when it is not empty, is formed by paths strictly contained in $\left(C^{(\alpha,v)}_s\right)^{\mu(\alpha)}a$ which don't start at $a$ but they do finish at $a$; whereas the collection {\footnotesize\begin{equation}\label{2.0040}\left\{\,\left(\overline{C^{(\alpha,v)}_s}\right)^j\bar{a}q^{\overline{(\alpha,w)}}_{r}\cdots q^{\overline{(\alpha,w)}}_{r+k-1}\,\Big{|}\,0\le j\le\mu(\alpha)-1,1\le k\le\occ(\alpha,W)-1\,\right\}\end{equation}} when it is not empty, is formed by prefixes of $\left(C^{(\alpha,v)}_s\right)^{\mu(\alpha)}a$ which start at $a$ but they don't finish at $a$. Particularly, the intersection of (\ref{2.0039}) and (\ref{2.0040}) is empty. Elements 
of these sets correspond to the appearing terms in (\ref{017}) and (\ref{016}). When the arrow $a$ is a loop and $\alpha\in\A_{\G}$, then the union of the sets {\footnotesize\begin{equation}\label{2.0041}\left\{\left(\overline{C^{(\alpha,v)}_s}\right)^j\overline{a}\,\Big{|}\,1\le j\le\mu(\alpha)-1\right\},\end{equation}}{\footnotesize\begin{equation}\label{2.0042}\left\{q^{\overline{(\alpha,v)}}_{s+l}\cdots q^{\overline{(\alpha,v)}}_{s-1}\left(\overline{C^{(\alpha,v)}_s}\right)^j\overline{a}\,\Big{|}\,1<l\le\occ(\alpha,V)-1,0\le j\le\mu(\alpha)-1\right\},\end{equation}}{\footnotesize\begin{equation}\label{2.0043}\left\{\left(\overline{C^{(\alpha,v)}_s}\right)^j\overline{a}q^{\overline{(\alpha,v)}}_{s+1}\cdots q^{\overline{(\alpha,v)}}_{s+k}\,\Big{|}\,1\le k<\occ(\alpha,V)-1,0\le j\le\mu(\alpha)-1\right\},\end{equation}}{\footnotesize\begin{equation}\label{2.0044}\left\{\left(\overline{C^{(\alpha,v)}_{s+1}}\right)^{j+1}-\left(\overline{C^{(\alpha,v)}_{s}}\right)^{j+1}\,\Big{|}\,0\le j\le\mu(\alpha)-2\right\},\end{equation}} is a linearly independent subset of the $K$-space $\L$. Elements of this set correspond to the appearing terms in (\ref{044}). And when $\alpha\in\B_{\G}$ also the set {\footnotesize\begin{equation}\label{2.0045}\left\{q^{\overline{(\alpha,v)}}_{s+l}\cdots q^{\overline{(\alpha,v)}}_{s},q^{\overline{(\alpha,v)}}_{s}\cdots q^{\overline{(\alpha,v)}}_{s+k}\,\Big{|}\,1< l\le\occ(\alpha,V)-1,1\le k<\occ(\alpha,V)-1\right\}\end{equation}} is a linearly independent subset of the $K$-space $\L$, and the elements in this set correspond to the appearing terms in (\ref{2.0038}).
\end{obs}

Now, let $\chi$ be our initial element in $\coprod_{v\in\Q_0}v\L v$ and we suppose that $\chi\in\textrm{ker}d_1^*$. Then we have \[\chi_{s(a)}\bar{a}-\bar{a}\chi_{t(a)}=0,\, \forall a\in\Q_1.\] From expressions in (\ref{017}) and (\ref{016}) we can say that \begin{equation}\label{018}x^{(s(a))}-x^{(t(a))}=0,\,\textrm{for all } a\in\Q_1.\end{equation} Because we are assuming that $\L$ is indecomposable it follows that the induced quiver $\Q$ is connected; since $\Q$ is formed by special cycles the expression in (\ref{018}) gives us that \begin{equation}\label{019}x^{(v)}=x^{(w)},\,\,\textrm{for all } v,w\in\Q_0.\end{equation} But  by (\ref{011}) we see that this implies that $\sum_{v\in\Q_0}v=1_{\L}$ is in the center of $\L$.

Now, let $\alpha\in\A_{\G}$ and $\olCC_{(\alpha)}=\left\{a^{(\alpha)}_{1},\ldots,a^{(\alpha)}_{\val(\alpha)}\right\}$ be the collection of arrows associated to $\alpha$, such that $t(a^{(\alpha)}_{i})=s(a^{(\alpha)}_{i+1})$ for each $1\le i\le\val(\alpha)$, and where $a^{(\alpha)}_{\val(\alpha)+1}=a^{(\alpha)}_{1}$. For an arbitrary arrow $a^{(\alpha)}_i$ in $\olCC_{(\alpha)}$, by looking at (\ref{017}) and (\ref{044}) we see that if $a^{(\alpha)}_{i}$ is either a loop or not a loop we obtain the  system of linear equations \begin{equation}\label{2.0048}y_{j,a^{(\alpha)}_{i}}^{(\alpha)}-y_{j,a^{(\alpha)}_{i+1}}^{(\alpha)}=0,\,\textrm{for all }1\le j\le\mu(\alpha)-1.\end{equation} That is, by considering each arrow in $\olCC_{(\alpha)}$ what we obtain is the following system of linear equations
{\footnotesize\begin{equation}\label{2.0049}
 \begin{matrix}
 y_{1,a^{(\alpha)}_{1}}^{(\alpha)}-y_{1,a^{(\alpha)}_{2}}^{(\alpha)}&=&0, & \cdots, & y_{1,a^{(\alpha)}_{\val(\alpha)}}^{(\alpha)}-y_{1,a^{(\alpha)}_{1}}^{(\alpha)}&=&0\\ y_{2,a^{(\alpha)}_{1}}^{(\alpha)}-y_{2,a^{(\alpha)}_{2}}^{(\alpha)}&=&0, & \cdots, & y_{2,a^{(\alpha)}_{\val(\alpha)}}^{(\alpha)}-y_{2,a^{(\alpha)}_{1}}^{(\alpha)}&=&0\\ \vdots &  & & &\vdots & &\\ 
 y_{\mu(\alpha)-1,a^{(\alpha)}_{1}}^{(\alpha)}-y_{\mu(\alpha)-1,a^{(\alpha)}_{2}}^{(\alpha)}&=&0,& \cdots, & y_{\mu(\alpha)-1,a^{(\alpha)}_{\val(\alpha)}}^{(\alpha)}-y_{\mu(\alpha)-1,a^{(\alpha)}_{1}}^{(\alpha)}&=&0
 \end{matrix}
\end{equation}
}Then we have that this linear system gives us
\begin{equation}\label{2.0050}y^{(\alpha)}_{j}=y_{j,a^{(\alpha)}_{i}}^{(\alpha)}=y_{j,a^{(\alpha)}_{i+1}}^{(\alpha)},\,\begin{array}{l}\forall1\le j\le\mu(\alpha)-1;\\ \forall1\le i\le\val(\alpha).\end{array}\end{equation} So, by returning to the expression (\ref{011}) we see that\[\sum_{j=1}^{\mu(\alpha)-1}\left(\sum_{C\in\CC_{(\alpha)}}y^{(\alpha)}_j\overline{C}^j\right)=\sum_{j=1}^{\mu(\alpha)-1}y^{(\alpha)}_j\left(\sum_{C\in\CC_{(\alpha)}}\overline{C}^j\right)\] is an element of the center of the algebra, and by Observation \ref{ob02} this element is equal to \begin{equation}\label{2.0051}\sum_{j=1}^{\mu(\alpha)-1}y^{(\alpha)}_jC(\alpha)^j.\end{equation} That is, for each $\alpha\in\A_{\G}$ we have that the element in (\ref{2.0051}) is in the center of the algebra.

\begin{obs}\label{ob03}
If $\alpha\in\C_{\G}$ then there exists a unique polygon $V$ such that $\V_{(\alpha)}=\left\{V\right\}$ and $\CC_{(\alpha)}=\left\{a^{(\alpha)}\right\}$. We proved in Proposition \ref{2.pr005}.\ref{2.pr005.1} that $\overline{a^{(\alpha)}}\in Z(\L)$, then $C(\alpha)^j=\left(\overline{a^{(\alpha)}}\right)^j\in Z(\L)$, for every $1\le j\le\mu(\alpha)-1$. So, for each $\alpha\in\C_{\G}$, we have that \begin{equation}\label{2.0052}\bigcup_{\alpha\in\C_{\G}}\left\{C(\alpha)^j\,\big{|}\,1\le j\le\mu(\alpha)-1\right\}\end{equation} is contained in $Z(\L)$.
\end{obs}

Recall that we are assuming that $\chi\in\textrm{ker}d_1^*$. Consider the case $\alpha\in\A_{\G}$ with $\olCC_{(\alpha)}=\left\{a^{(\alpha)}_1,\ldots,a^{(\alpha)}_{\val(\alpha)}\right\}$ the collection of arrows associated to $\alpha$. By Observation \ref{ob01} and expressions (\ref{017}) and (\ref{044}) we have that for each $1\le i\le\val(\alpha)$
\begin{itemize}
 \item if $a^{(\alpha)}_i$ is not a loop then \[y^{(\alpha)}_{j,a^{(\alpha)}_{i+1},k}=0,\begin{array}{l}\forall0\le j\le\mu(\alpha)-1;\\\forall1\le k\le\occ(\alpha,W)-1;\end{array}\]
 \item if $a^{(\alpha)}_i$ is a loop then \[y^{(\alpha)}_{j,a^{(\alpha)}_{i+1},k}=0,\begin{array}{l}\forall0\le j\le\mu(\alpha)-1;\\\forall1\le k<\occ(\alpha,W)-1;\end{array}\] and \[y^{(\alpha)}_{j,a^{(\alpha)}_{i+1},\occ(\alpha,V)-1}=0,\,\forall0\le j\le\mu(\alpha)-2.\]
\end{itemize}

From this we see that the only scalars that are not necessarily zero are those of the form \[y^{(\alpha)}_{\mu(\alpha)-1,a^{(\alpha)}_{i+1},\occ(\alpha,V)-1}\] when $a^{(\alpha)}_i$ is a loop, which are associated to the oriented cycles defined in (\ref{2.0021}). Now, if $\alpha\in\B_{\G}$ then by Observation \ref{ob01} and expressions (\ref{016}) and (\ref{2.0038}) we have that for each $1\le i\le\val(\alpha)$
\begin{itemize}
 \item if $a^{(\alpha)}_i$ is not a loop \[y^{(\alpha)}_{a^{(\alpha)}_{i+1},k}=0,\,\forall1\le k\le\occ(\alpha,W)-1;\]
 \item if $a^{(\alpha)}_i$ is a loop then \[y^{(\alpha)}_{a^{(\alpha)}_{i+1},k}=0,\,\forall1\le k<\occ(\alpha,V)-1.\]
\end{itemize}
Also at this case we see that the only scalars that are not necessarily equal to zero are those of the form \[y^{(\alpha)}_{a^{(\alpha)}_{i+1},\occ(\alpha,V)-1}\] when $a^{(\alpha)}_i$ is a loop, which also are associate to the cycles defined in (\ref{2.0021}).

Finally getting back to expression in (\ref{011}) and applying each of the previous computations, if $\chi$ is in ker$d_1^*$ then $\chi$ must be a linear combination of the elements of the following sets
\begin{itemize}
 \item $\left\{\,1_{\L}\,\right\}$
 \item $\bigcup_{\alpha\in\A_{\G}\cup\C_{\G}}\left\{\,C(\alpha)^j\,\big{|}\,1\le j\le\mu(\alpha)-1\,\right\}$
 \item$\left\{\,C^{(V)}\,\big{|}\,V\in\G_1\,\right\}$
 \item $\bigcup_{\alpha\in h^{-1}(1)}\left\{\overline{D^{(\alpha)}_{V,s}}\,\big{|}\,(V,s)\in\M^{(\alpha)}\right\}$
\end{itemize}
where $h$ is the map defined in (\ref{2.0018}) and $\M^{(\alpha)}$ the set defined in (\ref{2.0020}). The union of all sets above is disjoint and it is a subset of the $K$-basis of $\coprod_{v\in\Q_0}v\L v$ generating ker$d_1^*$. Hence is a $K$-basis of $Z(\L)$.\\

For the induced quiver $\Q$ of the Brauer configuration $\G$, let $\#\mathit{Loops}(\Q)$ denote the number of loops in the quiver $\Q$, then it is not difficult to prove that \begin{equation}\label{2.0053}\#\mathit{Loops}(\Q)=\sum_{\alpha\in h^{-1}(1)}\left(\sum_{V\in\V_{(\alpha)}}\left|\Psi_{(\alpha)}^{\,\,\,v}\right|\right)+|\C_{\G}|.\end{equation} See (\ref{2.0020}), (\ref{2.0021}) and (\ref{2.0022}) to check. We can say that the dimension of $Z(\L)$ is equal to \[1+\sum_{\alpha\in\A_{\G}\cup\C_{\G}}\left(\mu(\alpha)-1\right)+|\G_1|+\sum_{\alpha\in h^{-1}(1)}\left(\sum_{V\in\V_{(\alpha)}}\left|\Psi^{\,\,\,v}_{(\alpha)}\right|\right).\] Using the fact that $\mu(\alpha)-1=0$ for every $\alpha\in\B_{\G}\cup\T_{\G}$ and the expression in (\ref{2.0053}), the dimension can be expressed as
\begin{eqnarray*}
  & & 1+\sum_{\alpha\in\A_{\G}\cup\B_{\G}\cup\C_{\G}}(\mu(\alpha)-1)+|\G_1|+\#\mathit{Loops}(\Q)-|\C_{\G}|\\ & = & 1+\sum_{\alpha\in\G_0}(\mu(\alpha)-1)+|\G_1|+\#\mathit{Loops}(\Q)-|\C_{\G}|\\ & = & 1+\sum_{\alpha\in\G_0}\mu(\alpha)+|\G_1|-|\G_0|+\#\mathit{Loops}(\Q)-|\C_{\G}|
\end{eqnarray*}

Finally we obtain the desired result.

\begin{theo}\label{2.the001}
 Let $\L=K\Q/I$ be the Brauer configuration algebra associated to the connected and reduced Brauer configuration $\G$. Then\[\textrm{dim}_KZ(\L)=1+\sum_{\alpha\in\G_0}\mu(\alpha)+|\G_1|-|\G_0|+\#\mathit{Loops}(\Q)-|\C_{\G}|,\] where $\C_{\G}=\left\{\gamma\in\G_0\,|\,\val(\gamma)=1\textrm{ and }\mu(\gamma)>1\right\}$.
\end{theo}

\begin{defin}
 Let $\G=(\G_0,\G_1,\mu,\o)$ be a Brauer configuration. We say that $\G$ is a \textit{Brauer graph} if each polygon in $\G_1$ is a 2-gon. The induced algebra by $\G$ is called a \textit{Brauer graph algebra}.
\end{defin}

If the Brauer graph $\G=(\G_0,\G_1,\mu,\o)$ satisfies that the induced graph  $(\G_0,\G_1)$ is a tree, then we say that $\G$ is a \textit{Brauer tree} and the induced Brauer graph algebra a \textit{Brauer tree algebra}.

Let $\G$ be a Brauer tree and suposse that the induced graph $(\G_0,\G_1)$ is different of \begin{equation}\label{045}\xymatrix{\cdot\ar@(ur,ul)@{-}}\end{equation}Hence, in particular, $(\G_0,\G_1)$ has no loops. Now, if $\Q$ is the induced quiver by $\G$, then the only possible loops in $\Q$ are those induced by the vertices in $\C_{\G}$. This implies that \[\#\mathit{Loops}(\Q)=|\C_{\G}|.\] We have the following corollary.
\begin{cor}
 Let $\G=(\G_0,\G_1,\mu,\o)$ be a Brauer tree such that $(\G_0,\G_1)$ is not (\ref{045}). If $\L$ is the induced Brauer tree algebra, then
 \[\textrm{dim}_KZ(\L)=1+\sum_{\alpha\in\G_0}\mu(\alpha)+|\G_1|-|\G_0|.\]
\end{cor}
\begin{obs}
 By convention the algebra $k[x]/(x^2)$ is considered as a Brauer graph algebra. However, there is no Brauer configuration that induces this algebra.
\end{obs}

\section{Some examples}

In this section we consider some examples to apply the formula obtained previously.

\begin{ejem}
 Consider the unoriented graph 
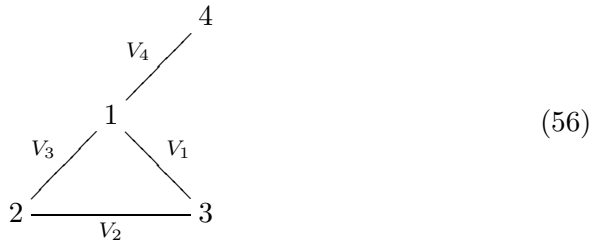
\begin{equation}\label{0029}\begin{split}\xymatrix{ & & 4 \\ & 1\ar@{-}[rd]^{V_1}\ar@{-}[ru]^{V_4} & \\2\ar@{-}[rr]_{V_2}\ar@{-}[ru]^{V_3} & & 3}\end{split}\end{equation} We can represent this unoriented graph by the ordered pair $(\G_0,\G_1)$ where $\G_0=\{\,1,2,3,4\,\}$, $\G_1=\{\,V_1,V_2,V_3,V_4\,\}$ and \[\begin{array}{rcl}V_1 & = & \left\{1,3\right\},\\V_2 & = & \left\{2,3\right\},\end{array}\begin{array}{rcl}V_3 & = & \left\{1,2\right\},\\V_4 & = & \left\{1,4\right\}.\end{array}\] If we define the multiplicity function as $\mu\equiv1$, and the orientation $\o$ for the nontruncated vertices given by \begin{eqnarray*}1 & : & V_1<V_3<V_4,\\2 & : & V_3<V_2,\\3 & : & V_1<V_2,\end{eqnarray*} we have that the tuple $\G=(\G_0,\G_1,\mu,\o)$ is a Brauer graph. Its induced quiver $\Q$ is given by 
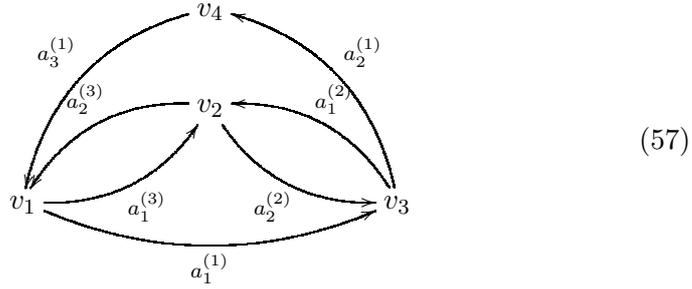
\begin{equation}\label{0030}\begin{split}\xymatrix{ & & v_4\ar@/_1.4pc/[lldd]_{a^{(1)}_3} & & \\ & & v_2\ar@/_1.1pc/[drr]_{a^{(2)}_2}\ar@/_1.1pc/[lld]_{a^{(3)}_2} & & \\v_1\ar@/_1.4pc/[rrrr]
_{a^{(1)}_1}\ar@/_1.1pc/[rru]_{a^{(3)}_1} & & & & v_3\ar@/_1.1pc/[ull]_{a^{(2)}_1}\ar@/_1.4pc/[uull]_{a^{(1)}_2}}\end{split}\end{equation} In this example we can see that $\B_{\G}=\{1,2,3\}$ and $\T_{\G}=\{4\}$, hence if $\L$ is the induced Brauer graph algebra we have that the value of the dimension of the center of $\L$ is equal to \begin{eqnarray*}\textrm{dim}_KZ(\L) & = & 1+\sum_{\alpha\in\G_0}\mu(\alpha)+|\G_1|-|\G_0|+\#\mathit{Loops}(\Q)-|\C_{\G}|\\ & = & 1+4+4-4+0-0\\ & = & 5\end{eqnarray*} and the elements of the basis of the center are given by
 \begin{itemize}
  \item $1_{\L}$
  \item $\begin{array}{rcccc}C^{(V_1)} & = & \overline{a^{(1)}_1a^{(1)}_2a^{(1)}_3} & = & \overline{a^{(3)}_1a^{(3)}_2},\\C^{(V_2)} & = & \overline{a^{(3)}_2a^{(3)}_1} & = & \overline{a^{(2)}_2a^{(2)}_1},\\C^{(V_3)} & = & \overline{a^{(1)}_2a^{(1)}_3a^{(1)}_1} & = & \overline{a^{(2)}_1a^{(2)}_2},\\C^{(V_4)} & = & \overline{a^{(1)}_3a^{(1)}_1a^{(1)}_2}. & &  \end{array}$
 \end{itemize}
\end{ejem}

\begin{ejem}
 In \cite{tai} the authors calculated the Hochschild cohomology ring for the Brauer graph algebra whose Brauer graph is given by a cycle with $m\ge1$ edges and $m$ vertices, and the multiplicity function is equal to $N$ at each vertex. We can represent this Brauer graph with the following Brauer configuration. First we assume that $m\ge2$. Let $\G=(\G_0,\G_1,\mu,\o)$ be the Brauer configuration given by
\begin{eqnarray*}
 \G_0 & = & \mathbb{Z}_m\\
 \G_1 & = & \left\{V_i\,|\,i\in\mathbb{Z}_m\right\}
\end{eqnarray*}
with $V_i=\left\{i,i+1\right\}$, where the multiplicity function is $\mu\equiv N$ and the orientation $\o$ is given by the succesor sequences \[i:V_{i-1}<V_i,\] for each $i\in\mathbb{Z}_m$. At this configuration we have that \[\C_{\G}=\left\{i\in\G_0\,|\,\val(i)=1\textrm{ and }\mu(i)>1\right\}=\emptyset,\] and in the induced quiver $\Q$ we also have that $\#\mathit{Loops}(\Q)=0$. If $\L$ es the induced Brauer configuration algebra we have that the dimension of the center is
\begin{eqnarray*}
 \textrm{dim}_KZ(\L) & = & 1+\sum_{i\in\G_0}\mu(i)+|\G_1|-|\G_0|+\#\mathit{Loops}(\Q)-|\C_{\G}|\\
  & = & 1+Nm+m-m+0-0\\
  & = & 1+Nm
\end{eqnarray*}
which is the same value in \cite[Theorem 3.1]{tai}. Now, when $m=1$ the Brauer configuration  is simply $\left(\{1\},V=\{1,1\},\mu\equiv N,\o\right)$, where the orientation $\o$ is given by the successor sequence $1:V<V$. We have that the induced quiver $\Q$ is 
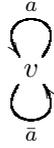
\[\xymatrix{v\ar@(ur,ul)_{a}\ar@(dl,dr)_{\bar{a}}}\] In this case we have that $\#\mathit{Loops}(\Q)=2$ and $\C_{\G}=\emptyset$, then if $\L$ is the induced Brauer graph algebra we have that the dimension value of the center of $\L$ is equal to \begin{eqnarray*}\textrm{dim}_KZ(\L) & = & 1+N+1-1+2-0\\ & = & N+3\end{eqnarray*} which coincides with the value given in \cite[Theorem 7.1]{tai}.
\end{ejem}

%\begin{xy}0;/r5pc/:*\dir{*};*\xycircle(1,1){++\dir{>}}\end{xy} 

\end{document}